\documentclass[11pt,reqno]{amsart}
\usepackage[margin=1in]{geometry}
\usepackage[utf8]{inputenc}
\usepackage{hyperref}
\usepackage{amsfonts,amsthm,amsbsy,amsmath,amssymb,verbatim}
\usepackage{latexsym,enumerate,enumitem}
\usepackage{mathrsfs} 
\usepackage{xcolor}

\usepackage{cleveref}

\hypersetup{
    colorlinks,
    linkcolor={red!50!black},
    citecolor={blue!50!black},
    urlcolor={blue!80!black}
}

\numberwithin{equation}{section}

\usepackage{thmtools}
\declaretheorem[numberwithin=section,name=Theorem]{theorem}
\declaretheorem[sibling=theorem,name=Lemma]{lemma}
\declaretheorem[sibling=theorem,name=Proposition]{proposition}
\declaretheorem[sibling=theorem,name=Corollary]{corollary}

\declaretheorem[sibling=theorem,name=Remark,style=remark]{remark}

\AtBeginDocument{\def\MR#1{}
}

\newcommand{\nbd}[2]{({#1})^{#2}}

\newcommand{\sph}[2]{S^{#1}({#2})}
\newcommand{\ft}{\widehat}

\newcommand{\A}{\mathcal C}
\newcommand{\E}{\mathbb E}

\renewcommand{\P}{\mathbb P}
\newcommand{\R}{\mathbb R}

\let\se\S
\renewcommand{\S}{\mathbb S}
\newcommand{\Z}{\mathbb Z}

\newcommand{\cH}{\mathcal{H}}

\DeclareMathOperator{\dist}{dist}

\newcommand{\les}{\lesssim}
\newcommand{\less}{\lessapprox}
\newcommand{\NS}{\mathcal{S}^\delta} \newcommand{\ST}{\mathcal{S}_{T}} 

\newcommand\labelrel[2]{\stackrel{\textnormal{\eqref{#2}}}{\mathstrut{#1}}}

\newcommand{\NM}{\mathcal{N}^\delta} \newcommand{\NMO}{\mathcal{N}^{O(\delta)}} \newcommand{\NT}{\mathcal{N}_{T}^{\delta}} \newcommand{\NTk}{\mathcal{N}_{T}^{2^{k+1}\delta}}

\newcommand{\MT}{\mathcal{N}_{T}}

\newcommand{\avg}{\operatorname{Avg}}

\newcommand{\ges}{\gtrsim}
\newcommand{\wh}{\widehat}

\usepackage{xcolor}

\usepackage{mathtools}
\DeclarePairedDelimiter\ceil{\lceil}{\rceil}
\DeclarePairedDelimiter\floor{\lfloor}{\rfloor}

\usepackage{caption}
\usepackage{subcaption}

\begin{document}
\title{Nikodym sets and maximal functions associated with spheres}
\author{Alan Chang}
\address{Department of Mathematics, Washington University in St.\ Louis, St.\ Louis, MO 63130, USA}
\email{alanchang@math.wustl.edu}

\author{Georgios Dosidis}
\address{Department of Mathematical Analysis, Faculty of Mathematics and Physics, Charles University,
			Prague, Czechia}
\email{dosidis@karlin.mff.cuni.cz}

\author{Jongchon Kim}
\address{Department of Mathematics, City University of Hong Kong, Hong Kong SAR, China}
\email{jongckim@cityu.edu.hk}

\begin{abstract} 
We study spherical analogues of Nikodym sets and related maximal functions. In particular, we prove sharp $L^p$-estimates for Nikodym maximal functions associated with spheres. As a corollary, any Nikodym set for spheres must have full Hausdorff dimension. In addition, we consider a class of maximal functions which contains the spherical maximal function as a special case. We show that $L^p$-estimates for these maximal functions can be deduced from local smoothing estimates for the wave equation relative to fractal measures. \end{abstract}

\maketitle

\tableofcontents

\section{Introduction} A set $A \subset \R^n$ of zero Lebesgue measure is a \emph{Nikodym set for spheres (resp., unit spheres)} if for every $y$ in a set of positive Lebesgue measure, there exists a sphere (resp., unit sphere) $S$ containing $y$ such that $A\cap S$ has positive $(n-1)$-dimensional Hausdorff measure. The existence of these sets was first proven in two dimensions in {\cite{CC2019}}. We extend the construction to all dimensions. (See \Cref{theorem:nikodym-translations-sphere}.)

Associated to these Nikodym sets are the following \emph{Nikodym maximal operators}:
\begin{align}
\label{eq:NM-def}
\NM f(x) &= \sup_{u\in \S^{n-1}} \frac{1}{|S^{\delta}(0)|} \left|\int_{S^\delta(0)} f(x+u+y)\, dy \right|
\\
\label{eq:NS-def}
\NS f(x) &= \sup_{u\in \S^{n-1}} \sup_{1\leq t\leq 2} \frac{1}{|S^{\delta}(0)|} \left|\int_{S^\delta(0)} f(x+t(u+y))\, dy \right|,
\end{align}
where $S^\delta(0)$ denotes the $\delta$-neighborhood of the unit sphere $\S^{n-1}$. $\NM f(x)$ is the supremum of averages of $f$ on the $\delta$-neighborhoods of every unit sphere passing through $x$, while for $\NS$ we allow the radius to vary. We show that these operators are bounded on $L^p(\R^n)$; furthermore, we show that in our bounds, the dependence on $\delta$ is sharp for every $p$ and $n$. (See \Cref{theorem:NM-upper-bounds,theorem:NS-upper-bounds}.) As a corollary, we conclude that all Nikodym sets for spheres have full Hausdorff dimension. (See \Cref{theorem:nikodym-dimension}.)

Restricting the set of translation to a compact $T \subset \R^{n}$, we also study the \emph{uncentered spherical maximal operators}:
\begin{align} 
\label{eq:MT-def}
\MT f(x)  &= \sup_{u\in T} \left| \int_{\S^{n-1}} f(x + u + y) \, d\sigma(y) \right |
\\
\label{eq:ST-def}
\ST f(x)  &= \sup_{u\in T} \sup_{t>0} \left| \int_{\S^{n-1}} f(x + t(u+y)) \, d\sigma(y) \right |,
\end{align}
where $d\sigma$ denotes the normalized surface measure on the sphere. These operators are initially defined for continuous functions with compact support. If $T = \S^{n-1}$, then the existence of Nikodym sets for unit spheres implies that these operators are not bounded on $L^p(\R^n)$ for any finite $p$. On the other hand, if $T=\{0\}$, $\ST$ is the classical spherical maximal function, which is known to be bounded for some range of $p$; see \Cref{sec:uncentered}. We show that if the upper Minkowski dimension of $T$ is strictly less that $n-1$, then $\MT$ and $\ST$ are bounded on $L^p(\R^n)$ for some range of $p$. (See \Cref{thm:MT,thm:ST}.)

\subsection{Nikodym maximal functions associated with spheres}

We start with a brief overview of the classical Nikodym sets and Nikodym maximal functions. A \emph{Nikodym set} (for lines) is a set $A \subset \R^n$ of zero Lebesgue measure such that for every $x \in \R^n$, there is a line $\ell$ through $x$ such that $A \cap \ell$ contains a unit line segment. The existence of Nikodym sets was discovered by Nikodym \cite{nikodym1927}. Nikodym sets are closely related to \emph{Kakeya sets} (a.k.a.~\emph{Besicovitch sets}), which are sets in $\R^n$ which contain a unit line segment in every direction. Kakeya sets of zero Lebesgue measure were discovered by Besicovitch \cite{Besicovitch1928}. 

In the seminal paper \cite{cordoba1977}, C\'ordoba introduced the Nikodym maximal function 
\begin{equation}\label{eq:classicalNikodymMaximal}
x \mapsto \sup_{\tau \ni x} \frac{1}{|\tau^\delta|}\left| \int_{\tau^\delta}f \right|,
\end{equation}
where the supremum is taken over all unit line segments $\tau$ centered at $x$, and $\tau^\delta$ denotes the $\delta$-neighborhood of $\tau$. 
Lower bounds on the dimension of Nikodym sets can be obtained from $L^p$-bounds for the Nikodym maximal function. See, e.g., \cite{Mattila2015} for more on Kakeya and Nikodym sets and the interplay between geometric measure theory and Fourier analysis.

The operators $\NM$ and $\NS$ defined above in \eqref{eq:NM-def} and \eqref{eq:NS-def} are analogues of \eqref{eq:classicalNikodymMaximal}, with spheres instead of lines. We now discuss the existence of Nikodym sets for spheres.

By adapting a construction of Cunningham \cite{cunningham}, H\'era and Laczkovich showed in \cite{HL2016} that a sufficiently short circular arc can be moved (via rigid motions) to any position in the plane within a region of arbitrarily small area; this can be considered the analogue of the Kakeya needle problem for circular arcs.

The results in \cite{HL2016} were extended by Cs\"ornyei and the first author, who showed that if one removes a neighborhood of two diametrically opposite points from a circle, the resulting set can be moved to any other position in arbitrarily small area \cite[Corollary 1.3]{CC2019}. In fact, they studied a Kakeya needle problem variant for all rectifiable sets, not just circles. See \cite[Theorem 1.2]{CC2019}.

Our first result is a higher-dimensional analogue of {\cite[Theorem 6.9]{CC2019}}, specialized to spheres.  

\begin{theorem}[Existence of Nikodym sets for unit spheres]
	\label{theorem:nikodym-translations-sphere}
	There exists a set $A \subset \R^n$ such that:
\begin{enumerate}
		\item
		$A$ has Lebesgue measure zero.
		\item
		For all $y \in \R^n$, there is a point $p_y \in \R^n$ and an $(n-1)$-plane $V_y $ containing $0$ such that  $y \in p_y + \S^{n-1}$  and $p_y + (\S^{n-1}  \setminus V_y) \subset A$.
	\end{enumerate}
Also, the mappings $y \mapsto p_y$ and $y \mapsto V_y$ are Borel.
\end{theorem}
The set $A$ from \Cref{theorem:nikodym-translations-sphere} contains, for every $y\in \R^n$,  a unit $(n-1)$-sphere containing $y$ up to a great $(n-2)$-sphere. It is a prototypical Nikodym set for unit spheres. The proof of \Cref{theorem:nikodym-translations-sphere} closely follows the proof in {\cite{CC2019}} for the case $n=2$, but we include an outline in \Cref{section:KakeyaSet} for the sake of completeness.

Due to the existence of a Nikodym set for unit spheres, for any finite $p$, the $L^p$-operator norm of the maximal function $\NM$ (see \eqref{eq:NM-def}) cannot be bounded uniformly in $\delta$ as $\delta \to 0$. The following theorem determines the $L^p$-operator norm $\| \NM \|_{L^p \to L^p}$ up to a factor of $\delta^{-\epsilon}$ for an arbitrarily small $\epsilon>0$. Here and in the following, we denote by $A\les B$ and $A\less B$ the inequalities $A\leq C B$ and $A\leq C_\epsilon \delta^{-\epsilon} B$ for any $\epsilon\in(0,1/2)$, respectively, for some absolute constants $C, C_\epsilon>0$. 

\begin{theorem}
	\label{theorem:NM-upper-bounds}
	\ 
	\begin{enumerate}[label=(\roman*)]
		\item\label{theorem:NM-upper-bounds-item:upper-bounds}
		When $n=2$, 
		\[
		\| \NM \|_{L^p(\R^2) \to L^p(\R^2)} \lessapprox 
		\begin{cases}
			\delta^{1 - \frac{2}{p}}, & \quad 1\leq p\leq 2
			\\
			1,& \quad 2\leq p\leq \infty.
		\end{cases}
		\] 
When $n = 3$, 
		\[ 
		\| \NM \|_{L^p(\R^3) \to L^p(\R^3)} \lessapprox 
		\begin{cases}
			\delta^{\frac{3}{2}-\frac{5}{2p}}, & 1\leq p\leq 3/2 
			\\
			\delta^{\frac{1}{2}-\frac{1}{p}}, &3/2 \leq p \leq 2
			\\
			1, & 2\leq p\leq \infty.
		\end{cases}
		\]
When $n \geq 4$, 
		\[ 
		\| \NM \|_{L^p(\R^n) \to L^p(\R^n)} \lessapprox 
		\begin{cases}
			\delta^{2-\frac{3}{p}}, & 1\leq p\leq 4/3 
			\\
			\delta^{\frac{1}{2}-\frac{1}{p}}, &4/3 \leq p \leq 2
			\\
			1, & 2\leq p\leq \infty.
		\end{cases}
		\]
		
		\item\label{theorem:NM-upper-bounds-item:sharpness}
		The powers of $\delta$ in part~\ref{theorem:NM-upper-bounds-item:upper-bounds} are sharp for all $n \geq 2$ and all $1 \leq p \leq \infty$.
	\end{enumerate}
\end{theorem}

Part~\ref{theorem:NM-upper-bounds-item:upper-bounds} of \Cref{theorem:NM-upper-bounds} is proved in \Cref{section:NM}, while part~\ref{theorem:NM-upper-bounds-item:sharpness} follows from  \Cref{prop:NDeltaLowerBd}. The proof of the upper bounds relies on geometric estimates on the intersections of $\delta$-annuli of unit spheres; see \Cref{sub:AppA}.

\begin{remark}
\label{remark:same-after-4}
In \Cref{theorem:NM-upper-bounds}, the sharp bound $\delta^{\frac{1}{2}-\frac{1}{p}}$ holds for $\frac{n}{n-1} \leq p \leq 2$ for dimensions $n=2,3,4$. From this observation, one might conjecture that this pattern continues in dimensions $n \geq 5$, but that is false. This is related to the following phenomenon in $\R^n$ for $n \geq 4$: the sets $A := \frac{1}{\sqrt{2}} \S^{1} \times \{0\}^{n-2}$ and $B := \{(0,0)\} \times \frac{1}{\sqrt{2}} \S^{n-3}$ have the property that every unit sphere centered at a point in $A$ contains all of the set $B$. This is explained in more detail in \Cref{prop:NDeltaLowerBd} and \Cref{remark:lenz}.
\end{remark}

Although a Nikodym set for unit spheres can be small in the sense of Lebesgue measure, it must be large in the sense of Hausdorff dimension. Indeed, by a standard argument (see, e.g., \cite[Lemma 11.9]{wolff2003lectures} or \cite[Theorem 22.9]{Mattila2015}), \Cref{theorem:NM-upper-bounds}\ref{theorem:NM-upper-bounds-item:upper-bounds}  for the case $p=2$ implies that the Hausdorff dimension of any Nikodym set for unit spheres in $\R^n$ must be $n$. In fact, we show the same is true for Nikodym sets for spheres (not necessarily associated with unit spheres):

\begin{theorem}\label{theorem:nikodym-dimension}
	Any Nikodym set for spheres in $\R^n$ must have the Hausdorff dimension $n$.
\end{theorem}

To prove \Cref{theorem:nikodym-dimension}, we need to consider a larger maximal function $\NS$ (see \eqref{eq:NS-def}), where each average is taken over the $\Theta(\delta)$-neighborhood 
of every sphere through $x$ of radius $t \in [1,2]$. 
Again, by the standard argument mentioned above, an estimate $\| \NS \|_{L^p(\R^n) \to L^p(\R^n)} \lessapprox 1$ for some $p<\infty$ implies \Cref{theorem:nikodym-dimension}. Thus, \Cref{theorem:nikodym-dimension} is a consequence of the following sharp estimates for the $L^p$-operator norm of $\NS$.

\begin{theorem}
	\label{theorem:NS-upper-bounds}
	\ 
	\begin{enumerate}[label=(\roman*)]
		\item\label{theorem:NS-upper-bounds-item:upper-bounds}
		When $n=2$, 
		\[
		\| \NS \|_{L^p(\R^2) \to L^p(\R^2)} \lessapprox 
		\begin{cases}
			\delta^{\frac{1}{2} - \frac{3}{2p}}, \quad &  1\leq p\leq 3
			\\
			1,& 3\leq p\leq \infty.
		\end{cases}
		\] 
When $n \geq 3$, 
		\[ 
		\| \NS \|_{L^p(\R^n) \to L^p(\R^n)} \lessapprox 
		\begin{cases}
			\delta^{1-\frac{2}{p}}, \quad &1 \leq p \leq 2
			\\
			1, & 2\leq p\leq \infty.
		\end{cases}
		\]
		
		\item\label{theorem:NS-upper-bounds-item:sharpness}
		The powers of $\delta$ in part~\ref{theorem:NS-upper-bounds-item:upper-bounds} are sharp for all $n \geq 2$ and all $1 \leq p \leq \infty$.
	\end{enumerate}
\end{theorem}

Part~\ref{theorem:NS-upper-bounds-item:upper-bounds} of \Cref{theorem:NS-upper-bounds} is proved in Section \ref{sec:NS}, while part~\ref{theorem:NS-upper-bounds-item:sharpness} follows from  \Cref{prop:NSLowerBd}.  In fact, we prove a more general result for maximal operators obtained by replacing the supremum over ${u\in \S^{n-1}}$ in the definition of $\NS$ by the supremum over ${u\in T}$ for compact sets $T\subset \R^n$ with finite $(n-1)$-dimensional upper Minkowski content (see \eqref{eqn:coverassume} for a definition). The same holds for the case of $\NM$ and \Cref{theorem:NM-upper-bounds}; see \Cref{remark:NT-same-bounds-as-NM}.

We sketch the proof of the $L^3(\R^2)$-estimate for $\NS$. By duality, it suffices to get a good bound on $\| \sum_{i} 1_{C_i^\delta} \|_{L^{3/2}(\R^2)}$ for a certain collection of circles $\{C_i\}$ of radius comparable to 1, where $C_i^\delta$ denotes the $\delta$-neighborhood of $C_i$. Under the identification of a circle $\{ x\in \R^2: |x-y|=t \}$ with the point $(y,t) \in \R^2 \times (0,\infty)$, we show that the collection of circles satisfies the following condition:
\begin{equation}\label{eqn:non-concentration}
   \#\{ i: C_i \in B \} \les \delta^{-2} r \qquad\text{for any ball $B\subset \R^3$ of radius $r\geq \delta$}
\end{equation}
(see \Cref{lem:count}). The $L^{3/2}$ estimate is then covered by a recent result of Pramanik-Yang-Zahl \cite[Theorem 1.7]{Pramanik-Yang-Zahl}. See \Cref{prop:M2} for a slightly more general version of their result specialized to circles to be used for the proof, and see \Cref{thm:avg_frac} for a version of \Cref{prop:M2} in terms of the wave equation.

We note that the maximal function studied by Kolasa and Wolff \cite{Kol_Wol, Wolff_Circ} 
\begin{equation}\label{eqn:Wdelta}
  W^\delta f(t) = \sup_{x\in \R^n} \frac{1}{|S^{\delta}(0)|} \int_{S^\delta(0)} |f(x- t y)|dy, \; t\in [1,2]  
\end{equation}
satisfies the bound $\| W^\delta \|_{L^p(\R^n) \to L^p([1,2])} \lessapprox  1$ precisely when $\| \NS \|_{L^p(\R^n) \to L^p(\R^n)} \lessapprox  1$. This is not a coincidence. It has to do with the fact that a duality argument for $W^\delta$ also produces a ``one-dimensional" set of circles $\{C_i\}$ in the sense that it satisfies \eqref{eqn:non-concentration} with $\delta^{-1} r$ in place of $ \delta^{-2} r $.\footnote{The ``one" in ``one-dimensional" refers to the exponent of $r$ in \eqref{eqn:non-concentration}} In \Cref{sec:maxspheres}, we discuss a general maximal theorem associated with spheres which covers both $W^\delta$ and $\NS$ under the same framework. Regarding Wolff's $L^3$-bound for $W^\delta$ for the case $n=2$ \cite{Wolff_Circ}, see also \cite{WoL, Schlag_incidence, Zahl_Circular, Zahl, Pramanik-Yang-Zahl} for different proofs and extensions to more general curves.

\subsection{Uncentered spherical maximal functions}\label{sec:uncentered}
For a compact set $T\subset \R^n$, recall the definition of $\ST$ given in \eqref{eq:ST-def}. When $T=\{0\}$, $\ST$ is the (classical) spherical maximal function,
\begin{align}
\label{eq:spherical-maximal-Sf}
Sf(x) = \sup_{t>0} \left| \int_{\S^{n-1}} f(x + ty) \, d\sigma(y) \right |, 
\end{align}
which is known to be bounded on $L^p(\R^n)$ if and only if $p>\frac{n}{n-1}$, thanks to the seminal works of Stein \cite{Ste_Max} (when $n\geq 3$) and Bourgain \cite{Bou_Max} (when $n=2$). For $n=2$, the same bound holds for the case $T=\{u\}$ for any $u \in S^{1}$, which is a consequence of Sogge’s generalization \cite{Sogge_lo} of Bourgain’s circular maximal theorem.

If $T= \S^{n-1}$, $\ST$ is not bounded on $L^p(\R^n)$ for any $p<\infty$ due to the existence of Nikodym sets for unit spheres. 
In fact, for the same reason, the maximal function $\MT$ (see \eqref{eq:MT-def}), defined without the supremum over $t>0$, is unbounded on $L^p(\R^n)$ for any $p<\infty$.

Some positive results can be obtained for both $\MT$ and $\ST$ when $T$ is in between these two extreme cases, with the range of boundedness dependent on the dimension of $T$. In particular, we show that $\MT$ and $\ST$ are bounded on $L^p$ for some finite $p$ if there exists some $0\leq s<n-1$ such that $T$ has \emph{finite $s$-dimensional upper Minkowski content}, i.e.,
\begin{equation}\label{eqn:coverassume}
	N(T,\delta) \les \delta^{-s}, \, \text{ for all } \, \delta\in (0,1/2)
\end{equation}
where $N(T,\delta)$ denotes the $\delta$-covering number, the minimal number of balls of radius $\delta$ needed to cover $T$. 

\begin{remark}
The condition \eqref{eqn:coverassume} implies that $T$ has upper Minkowski dimension at most $s$. Conversely, if the upper Minkowski dimension of $T$ is $d$, then \eqref{eqn:coverassume} holds for all $s < d$. See, e.g., \cite[Chapter 5]{mattila95} for the definitions of upper Minkowski dimension and content. We do not need them in this paper.
\end{remark}

\begin{theorem}\label{thm:MT}
	Let $n\geq 2$ and $0\leq s<n-1$. Suppose that $T\subset \R^n$ is a compact set with finite $s$-dimensional upper Minkowski content (see \eqref{eqn:coverassume}). Then $\MT$ is bounded on $L^p(\R^n)$ 
\begin{enumerate}[label=(\roman*)]
	\item when $n=2$ and  
	\[p> 1 + s,\]
	\item when $n=3$ and 
    \[ p> 1+ \min\left(\frac{s}{2}, \frac{1}{3-s}, \frac{5-2s}{9-4s}\right), \]
\item when $ n\geq 4$ and
    \[ p> 1+ \min\left(\frac{s}{n-1},\frac{1}{n-s}, \frac{n-s}{3(n-s)-2}\right). \]
\end{enumerate}
\end{theorem}

We prove \Cref{thm:MT} in \Cref{section:NT}. 
The following proposition, proved in \Cref{sec:lowerboundMT}, contains necessary conditions for the boundedness of $\MT$.

\begin{proposition}\label{prop:lowerboundMT}
	Let $n\geq 2$ and $0\leq s <n-1$. Suppose that $\MT$ is bounded on $L^p(\R^n)$ for all compact sets $T$ satisfying \eqref{eqn:coverassume}. Then 
	\begin{equation*}
		p \geq 1+  
		\begin{cases}
			\frac{s}{n-1}, & 0\leq s\leq 1 \\
			\max\left(\frac{1}{n-1},\frac{s}{2n-3}, 3 - 2(n-s)\right),  & 1 < s \leq 2 \\
			\max\left(\frac{2}{2n-3},\frac{s+1-\ceil{s}}{n+1-\ceil{s}}, \frac{1}{n-\lfloor s\rfloor+1}, 3 - 2(n-s)\right), & 2 <  s < n-1.
		\end{cases}
	\end{equation*}
\end{proposition} 

\begin{remark}
For the range $1<s\leq 2$, the term $3-2(n-s)$ is negative for $n \geq 4$ and is relevant only for $n=3$.
\end{remark}

We note that the range of boundedness in \Cref{thm:MT} does not match with the necessary conditions of \Cref{prop:lowerboundMT} in general, but they do agree for the range $0\leq s\leq 1$ up to the endpoint. In particular, the bound for $n=2$ in \Cref{thm:MT} is essentially sharp. In \Cref{fig:MT}, we graph the two ranges in the case $n=5$ as an example.

\begin{figure}[h]
\centering
\includegraphics[page=1]{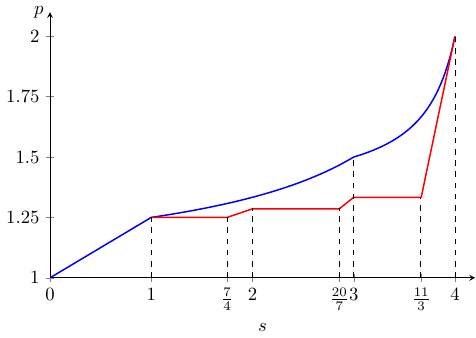}
\caption{Range of boundedness of $\MT$ in the case $n=5$. The blue and red lines indicate the boundaries of sufficient and necessary conditions, respectively.}
\label{fig:MT}
\end{figure}

In the case of $\ST$, which also includes dilations, we have the following bounds. 

\begin{theorem}\label{thm:ST} Let $n\geq 2$ and $0\leq s<n-1$. Suppose that $T\subset \R^n$ is a compact set with finite $s$-dimensional upper Minkowski content (see \eqref{eqn:coverassume}). Then $\ST$ is bounded on $L^p(\R^n)$ for
	\begin{align*}
		p >  
		\begin{cases}
			2+\min\left(1,\max\left(s,\frac{4s-2}{2-s}\right) \right), &n=2\\
			1+ \left[ n-1-s + \max(0, \min(1, (2s-n+3)/{4}))\right]^{-1}, &n\geq 3.
		\end{cases}
	\end{align*}
\end{theorem}
We prove \Cref{thm:ST} in \Cref{sec:duality}.
We examine the lower bounds for $p$ given in \Cref{thm:ST} in two extreme cases: 
\begin{enumerate}
\item For $s = 0$, the lower bound for $p$ is $\frac{n}{n-1}$. This matches the critical exponent for the boundedness of the spherical maximal function $S$ (defined in \eqref{eq:spherical-maximal-Sf}).
\item In the limit $s \to n-1$, the lower bound for $p$ converges to $3$ for $n=2$ and $2$ for $n\geq 3$. This matches the critical exponents in \Cref{theorem:NS-upper-bounds}.
\end{enumerate}

We also obtain, in \Cref{section:ST-lower}, the following necessary conditions for $\ST$ to bounded on $L^p(\R^n)$.

\begin{proposition}\label{cor:lower}
	Let $n\geq 2$ and $0\leq s <n-1$. Suppose that $\ST$ is bounded on $L^p(\R^n)$ for all compact sets $T$ with finite $s$-dimensional upper Minkowski content. Then 
\begin{align*}
	p \geq 1+
    \begin{cases}
        \max(\frac{ 1-(\ceil{s}-s)/2 }{n-(\ceil{s}+2)/2} , \frac{1}{n-\floor{s}/2} ), &s \leq 2\\
        \max(\frac{1+s-\ceil{s}}{n-\ceil{s}}, \frac{1}{n-\floor{s}} ), & s \geq 2. 
    \end{cases}
\end{align*}
\end{proposition}

In general, the sufficient conditions in \Cref{thm:ST} for the $L^p$-boundedness of $\ST$ do not match the necessary conditions in \Cref{cor:lower}, but there are a few cases where they do match (modulo the endpoint). The sufficient conditions are sharp modulo endpoint for $n=2$ and $0\leq s \leq \sqrt{3}-1$ and for $n=3$ and $s=1$.  Another case when the two bounds match is when $s$ is an integer greater than or equal to $\frac{n+1}2$. This can be observed in \Cref{fig:ST} when $s=3$, where we graph the range of boundedness for $\ST$ when $n=5$.

\begin{figure}[h]
\centering
\includegraphics[page=2]{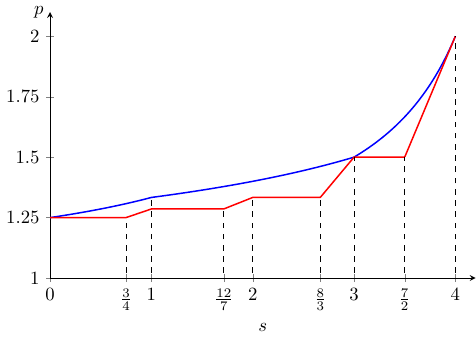}
\caption{Range of boundedness of $\ST$ in the case $n=5$. The blue and red lines indicate the boundaries of sufficient and necessary conditions, respectively.}
\label{fig:ST}
\end{figure}

The boundedness of the spherical maximal function $S$ (defined in \eqref{eq:spherical-maximal-Sf})
implies that there is no set of Lebesgue measure zero in $\R^n$ containing a sphere centered at every point in $\R^n$. This was proven independently for $n=2$ by Marstrand {\cite{Marstrand1987PackingCI}}. This is in contrast to the existence of Nikodym sets for spheres. Similarly, \Cref{thm:ST} has the following geometric consequence. We refer the reader to \cite[Chapter XI, Section 3.5]{Stein} for the implication.  
\begin{corollary}\label{cor:CDK3a} Let $n\geq 2$, $A\subset \R^n$ be a set of Lebesgue measure zero and $T\subset \R^n$ be a compact set  with finite $s$-dimensional upper Minkowski content for some $0 \leq s < n-1$. Then for almost every $y\in \R^n$, the $(n-1)$-dimensional Hausdorff measure of the set $A\cap (y+t(u+\S^{n-1}))$ is zero for every $t>0$ and every $u\in T$. 
\end{corollary}
 Corollary \ref{cor:CDK3a} is sharp in the sense that it fails to hold for $T=\S^{n-1}$ satisfying \eqref{eqn:coverassume} for $s=n-1$ due to the existence of Nikodym sets for spheres. Indeed, for a set $A$ and a map $y \mapsto p_y$ as in  \Cref{theorem:nikodym-translations-sphere}, the set $A\cap (y+u_y+\S^{n-1})$ has positive $(n-1)$-dimensional Hausdorff measure where $u_y=p_y-y \in \S^{n-1}$.

Although \Cref{cor:CDK3a} follows from \Cref{thm:ST}, it does not require the full strength of it. Indeed, it is possible to obtain $L^p$-bounds for a smaller range of exponents $p$ by using the local smoothing estimates for the wave equation, which is enough to deduce \Cref{cor:CDK3a}. We learned this observation from the authors of \cite{HKL}. There is a related result due to Wolff \cite{WoL} and Oberlin \cite{Oberlin} (see also \cite{Mitsis}): a Borel set containing a set of spheres of Hausdorff dimension larger than $1$ (as a subset of $\R^n \times \R_+$) must have positive Lebesgue measure. 

Next, we discuss the proof of \Cref{thm:ST}. It is well-known by the work \cite{MSS_Bo} that bounds for the circular maximal function $S$ can be deduced from local smoothing estimates for the wave equation. For the case of $\ST$, we observe a similar connection between $L^p$-bounds and fractal local smoothing estimates for the wave equation. To describe fractal local smoothing estimates, we fix some notation. 
For each $0<\alpha\leq n+1$, we denote by $\mathcal{C}(\alpha)$ the class of non-negative Borel measures $\nu$ supported on $\R^n \times [1,2]$ such that 
\begin{equation}\label{eqn:alpha_dimensional_measure}
\nu(B_r)  \leq r^\alpha,    
\end{equation}
for any ball $B_r\subset \R^{n+1}$ of radius $r$ for any $0<r\leq 1$. 

Suppose that $u(x,t)$ is the solution to the wave equation with initial data $u(\cdot,0) = f$ and $\partial_t u(\cdot,0) = 0$. 
A version of fractal local smoothing estimates for the wave equation is concerned with the ratio between $\| u \|_{L^q(\R^n \times [1,2],\nu)}$ and the $L^p(\R^n)$ norm of $f$ whose Fourier transform is supported on the annulus $\{ \xi \in \R^n: |\xi| \sim \delta^{-1} \}$ for $\delta \in (0,1)$. Even when $\nu$ is a constant multiple of the Lebesgue measure, this is a  difficult open problem for $n\geq 3$, which was settled for the case $n=2$ by Guth--Wang--Zhang \cite{GWZ}. We refer the reader to \cite[Section 1]{GWZ} for earlier results for $n=2$ and partial results in higher dimensions. For fractal measures, the problem has been studied in \cite{Wolff_decay, Erdogan, CHL, Harris, Harris0} for the case $p=2$, and sharp results are known for $n=2,3$. For the case $p>2$, Ham--Ko--Lee \cite{HKL} obtained some sharp estimates and used them to prove $L^p$-improving estimates for the circular maximal function relative to a class of fractal measures as a corollary. 
To be precise, the class of measures considered in \cite{HKL} is different from $\mathcal{C}(\alpha)$ in that \eqref{eqn:alpha_dimensional_measure} is required for every $r>0$, but this does not really change the problem since the local problem on $B_1 \times [1,2]$ is equivalent to the global problem on $\R^n \times [1,2]$ for estimating $u(x,t)$; see e.g. \cite[Lemma 2.6]{HKL}. 
We use results obtained in the papers \cite{CHL,HKL} for the proof of \Cref{thm:ST}, which is summarized in \Cref{thm:HKL}.

For the $n=2$ case, we need an additional ingredient: \Cref{prop:M2} (cf.\ \cite[Theorem 1.7]{Pramanik-Yang-Zahl}) discussed earlier. We reformulate the result in terms of the wave equation and spherical means.

\begin{theorem}\label{thm:avg_frac}
	Let $n=2$ and $\nu \in \mathcal{C}(\alpha)$ for $\alpha=1$. Suppose that $u(x,t)$ is the solution to the wave equation with initial data $u(\cdot,0) = f$ and $\partial_t u(\cdot,0) = 0$. If the Fourier transform of $f$ is supported on $\{\xi\in \R^2: |\xi| \sim \delta^{-1} \}$ for some $\delta \in (0,1/2)$, then
\begin{equation}\label{eqn:fraclo}
    \| u \|_{L^3(\nu)} \less \delta^{-1/2} \| f\|_{L^3(\R^2)}.
\end{equation}
 
Moreover, if $p>3$ and $\nu \in \mathcal{C}(\alpha)$ for some $\alpha>1$, then 
\begin{equation}\label{eqn:avg_frac}
	\| \avg f \|_{L^p(\nu)} \les \| f\|_{L^p(\R^2)}
\end{equation}
for $f\in L^p(\R^2)$, where $\avg f (x,t)$ denotes the  average of $f$ over the sphere of radius $t$ centered at $x$.
\end{theorem}
The estimate \eqref{eqn:fraclo} can be stated in terms of the inhomogeneous Sobolev norm:
 \[
	\| u \|_{L^3(\nu)} \les \| f\|_{3,{\frac{1}{2}+\epsilon}},
 \]
for any $\epsilon>0$. It implies a version of Wolff's circular maximal theorem 
\[ \| u \|_{L^3_t([1,2], L^\infty_x(\R^2))} \les \| f\|_{3,{\frac{1}{2}+\epsilon}}, \]
from \cite[Equation (2)]{Wolff_Circ} by a linearlization argument; see \cite{HKL}. 

In the paper \cite{HKL}, it was conjectured that \eqref{eqn:avg_frac} holds for any $p>4-\alpha$ when $1<\alpha\leq 2$, which was proved for $\alpha \geq 3-\sqrt{3}$. 
This conjecture would imply $L^p(\R^2)$-estimates for the maximal function $\ST$ for compact $T\subset \R^2$ with finite $s$-dimensional upper Minkowski content for every $0\leq s<1$ and $p>2+s$, which would be sharp except possibly for endpoint.

\subsection*{Organization of the article} 
We organize the paper as follows. We study the maximal functions $\NM$ and $\NS$ and prove \Cref{theorem:NM-upper-bounds}\ref{theorem:NM-upper-bounds-item:upper-bounds} and \Cref{theorem:NS-upper-bounds}\ref{theorem:NS-upper-bounds-item:upper-bounds} in Section \ref{sec:NM} and \Cref{sec:NS}, respectively. We prove \Cref{thm:MT} and \Cref{thm:ST} in Section \ref{sec:ST}. Section \ref{sec:duality} contains a discussion on a geometric approach to fractal local smoothing estimates and the proof of \Cref{thm:avg_frac}. We prove lower bounds for maximal functions considered in this paper in Section \ref{sec:LowerBound}. In \Cref{section:KakeyaSet}, we sketch the proof of \Cref{theorem:nikodym-translations-sphere} and in \Cref{section:volumebound}, we prove volume bounds for the  intersection of annuli used to study $\NM$ and $\NS$.

\subsection*{Acknowledgments}

We thank David Beltran, Loukas Grafakos, Sanghyuk Lee, Malabika Pramanik, Andreas Seeger, Tongou Yang, Joshua Zahl for helpful discussions. We thank Joshua Zahl for pointing us to \cite{Pramanik-Yang-Zahl} and for sketching the proof of \Cref{prop:M2}. We thank the anonymous referees for their helpful comments and suggestions.

G.D.\ was partially supported by the Primus research programme PRIMUS/21/SCI/002 of Charles University.
J.K.\ was partially supported by a grant from the Research Grants Council of the Hong Kong Administrative Region, China (Project No. CityU 21309222). 

Part of the work was done while A.C.\ and G.D.\ were visiting the Hausdorff Research Institute for Mathematics in Bonn during the research trimester ``Interactions between Geometric measure theory, Singular integrals, and PDE,'' which was funded by the Deutsche Forschungsgemeinschaft (DFG, German Research Foundation) under Germany's Excellence Strategy -- EXC--2047/1 -- 390685813.

\section{The Nikodym maximal function $\NM$ associated with unit spheres}\label{sec:NM}

\label{section:NM}

In this section we prove \Cref{theorem:NM-upper-bounds}\ref{theorem:NM-upper-bounds-item:upper-bounds}. It suffices to show prove that
\begin{align}
\label{eq:NM-L1}
\|\NM\|_{L^1 \to L^1} &\lesssim \delta^{-1} & (n \geq 2)
\\
\label{eq:NM-Linfty}
\|\NM\|_{L^\infty \to L^\infty} &\leq 1 & (n \geq 2)
\\
\label{eq:NM-L2}
\|\NM\|_{L^2 \to L^2} &\lesssim (\log(1/\delta))^{1/2} & (n \geq 2)
\\
\label{eq:L32-bound}
\|\NM\|_{L^{3/2} \to L^{3/2}} &\lesssim \delta^{-1/6}(\log(1/\delta))^{1/3} & (n \geq 3)
\\
\label{eq:L43-bound}
\|\NM\|_{L^{4/3} \to L^{4/3}} &\lesssim  \delta^{-1/4}(\log(1/\delta))^{1/4} & (n \geq 4)
\end{align}
The bounds \eqref{eq:NM-L1} and \eqref{eq:NM-Linfty} are straightforward, so it suffices to prove the remaining three bounds.

\subsection{Reduction to geometric estimates}

In this subsection, we reduce \Cref{theorem:NM-upper-bounds}\ref{theorem:NM-upper-bounds-item:upper-bounds} to the following geometric lemma.

\begin{lemma}
\label{lemma:geometric-estimates-intersections-234}
There exists an absolute constant $c_{diam} > 0$ such that the following is true. Let $n\geq 2$, $\delta \in (0,1/2)$. Let $\{S_i\}_{i \in I}$ be a collection of unit spheres in $\R^n$ whose centers lie in a set of diameter at most $c_{diam}$. Suppose that for each $i$, there exists $\omega_i \in S_i$ such that $\{\omega_i\}_{i \in I}$ is a $\delta$-separated subset. Then we have the following estimates. 

For $n \geq 2$,
\begin{align}
\tag{2SPH}
\label{eq:sum-j-desired-bound}
\delta^{n}
\sum_{\substack{
j \in I
}}
|S_i^\delta \cap S_j^\delta|
\lesssim
\delta^2
\log(1/\delta)
\qquad\text{for each $i \in I$}.
\end{align}

Furthermore, for $n \geq 3$,
\begin{align}
\tag{3SPH}
\label{eq:sum-jk-desired-bound}
\delta^{2n}
\sum_{\substack{
j,k \in I
}}
|S_i^\delta \cap S_j^\delta \cap S_k^\delta|
\lesssim
\delta^{5/2} \log(1/\delta)
\qquad\text{for each $i \in I$}.
\end{align}

Furthermore, for $n \geq 4$,
\begin{align}
\tag{4SPH}
\label{eq:sum-jkl-desired-bound}
\delta^{3n}
\sum_{\substack{
j,k,\ell \in I
}}
|S_i^\delta \cap S_j^\delta \cap S_k^\delta \cap S_\ell^\delta|
\lesssim
\delta^{3} \log(1/\delta)
\qquad\text{for each $i \in I$}.
\end{align}
\end{lemma}

\begin{remark}
\label{remark:average-intersection}
Since $\# I \lesssim \delta^{-n}$, we can view \eqref{eq:sum-j-desired-bound} as a bound on the average size of $|S_i^\delta \cap S_j^\delta|$, taken over all $j$. (Recall $i$ is fixed.) Thus, ignoring logarithmic factors, \eqref{eq:sum-j-desired-bound} says that in the worst case configuration, the ``typical'' $|S_i^\delta \cap S_j^\delta|$ is at most of size $\delta^2$. Similarly, we can view the left-hand sides of \eqref{eq:sum-jk-desired-bound} and \eqref{eq:sum-jkl-desired-bound} as averages (taken over all $j,k$ and over $j,k,\ell$, respectively). In the worst case configuration, the typical volume of the intersection is at most of size $\delta^{5/2}$ and $\delta^3$, respectively.
\end{remark}

\begin{remark}
\label{remark:mSPH}
With the same assumptions as in \Cref{lemma:geometric-estimates-intersections-234}, we have the following estimate for all dimensions $n \geq 4$ and all $m \geq 4$:
\begin{align}
\label{eq:mSPH}
\tag{mSPH}
\delta^{(m-1)n}
\sum_{\substack{
j_1, \ldots, j_{m-1} \in I
}}
|S_i^\delta \cap S_{j_1}^\delta \cap \cdots \cap  S_{j_{m-1}}^\delta|
\lesssim
\delta^{3} \log(1/\delta)
\qquad\text{for each $i \in I$}.
\end{align}
This follows immediately from the trivial bound $|S_i^\delta \cap S_{j_1}^\delta \cap \cdots \cap  S_{j_{m-1}}^\delta| \leq |S_i^\delta \cap S_{j_1}^\delta \cap S_{j_2}^\delta \cap S_{j_3}^\delta|$, the estimate \eqref{eq:sum-jkl-desired-bound}, and $\# I^{m-4} \lesssim (\delta^{-n})^{m-4}$. The power of $3$ in \eqref{eq:mSPH} cannot be improved; this is related to \Cref{remark:same-after-4} and \Cref{remark:lenz}. (The estimate \eqref{eq:mSPH} is not needed to prove \Cref{theorem:NM-upper-bounds}\ref{theorem:NM-upper-bounds-item:upper-bounds}.)
\end{remark}

To prove the reduction, we will need the following standard duality argument (see, e.g., \cite[Proposition 22.4]{Mattila2015}). For the convenience of the reader, we include a proof.
\begin{lemma}[Duality] \label{lem:duality-unit-radius}
Let $1 \leq p,q\leq \infty$. Let $\{ \omega_i \}$ be a maximal $\delta$-separated set of points in $\R^n$. Then 
\[ \| \NM \|_{L^p(\R^n) \to L^q(\R^n)} \les \delta^{n-1} \sup \| \sum_{i} a_i 1_{S_i^{O(\delta)}} \|_{L^{p'}(\R^n)}. \]
where the supremum is taken over all choices of unit spheres $S_i \ni \omega_i$ and $a_i\geq0$ such that $\delta^n \sum_{i}  a_i^{q'} = 1$. The implied constant in $O(\delta)$ is absolute.
\end{lemma}
\begin{proof} 
Let $f \in L^p(\R^n)$. We make the following observation:  if $|x - x'| \leq \delta$ and $S$ is any unit sphere containing $x$, then there exists a unit sphere $S'$ containing $x'$ such that $S^\delta \subset (S')^{2\delta}$. Thus,
\begin{align}
\label{eq:nearby-NM-bound}
\text{If } |x - x'| \leq \delta, \text{ then } \NM f(x) \lesssim \NMO f(x').
\end{align}
Combining this with duality of $\ell^q$ spaces and H\"older's inequality, 
\begin{align*}
\| \NM f \|_{L^{q}(\R^n)}
&\labelrel\les{eq:nearby-NM-bound}
\left(\sum_{i}  \delta^n \NMO f(\omega_i)^{q}\right)^{1/q}
\\
&=
\sup 
\sum_i \delta^n a_i \NMO f(\omega_i)
\\
&\les
\delta^{n-1}
\sup
\sum_i a_i \int_{\R^n} |f| 1_{S_i^{O(\delta)}}
\\
&\leq
\delta^{n-1}
\sup 
\| \sum_{i} a_i 1_{S_i^{O(\delta)}} \|_{L^{p'}(\R^n)}
\| f \|_{L^p(\R^n)}
\end{align*}
where the supremum is taken over all choices of unit spheres $S_i \ni \omega_i$ and $a_i\geq0$  such that $\delta^n \sum_{i}  a_i^{q'} = 1$.
\end{proof}

\begin{proof}[Proof that \eqref{eq:sum-j-desired-bound} implies \eqref{eq:NM-L2}]

Let $\{ \omega_i \}_{i \in I}$ be a maximal $\delta$-separated set of points in $\R^n$. For each $i$, let $S_i$ be a unit sphere containing $\omega_i$, and let and $a_i\geq0$ such that 
\begin{align}
\label{eq:sum-to-1}
\delta^n \sum_{i \in I}  a_i^{q'} = 1
\end{align}
By \Cref{lem:duality-unit-radius}, it is enough to show 
\begin{align}
\label{eq:2sph-duality-sufficient}
\|  \sum_{i \in I} a_i 1_{S_i^\delta} \|_{L^2(\R^n)}^2
\lesssim
\delta^{2-2n} \log(1/\delta).
\end{align}

We partition $\R^n = \bigcup_m P_m$ such that each $P_m$ has diameter at most $c_{diam}$ and the $2$-neighborhoods of $P_m$ have bounded overlap (depending only on $n$). We let $I_m$ be the set of $i \in I$ such that the center of $S_i$ is in $P_m$. Bounded overlap implies that 
\begin{align}
\label{eq:bounded-overlap}
\# \{m : \sum_{i \in I_m} a_i 1_{S_i^\delta}(x) \neq 0 \} \lesssim_n 1
\qquad\text{ for each } x \in \R^n.
\end{align}
This gives
\begin{align*}\|  \sum_{i \in I} a_i 1_{S_i^\delta} \|_{L^2(\R^n)}^2
&\labelrel\lesssim{eq:bounded-overlap}
\sum_m \|  \sum_{i \in I_m} a_i 1_{S_i^\delta} \|_{L^2(\R^n)}^2
\\
&= 
\sum_m \sum_{i,j \in I_m}  a_{i} |S_i^\delta \cap S_j^\delta|^{1/2} a_{j}  |S_i^\delta \cap S_j^\delta|^{1/2} 
\\
&\leq 
\sum_m \sum_{i,j \in I_m}  a_{i}^2 |S_i^\delta \cap S_j^\delta| 
\\
&\labelrel\lesssim{eq:sum-j-desired-bound}
\delta^{2-n} \log(1/\delta) \sum_m  \sum_{i \in I_m}  a_{i}^2  
\\
&\labelrel={eq:sum-to-1}
\delta^{2-2n} \log(1/\delta),
\end{align*}
where we used Cauchy-Schwarz in the middle inequality. This proves \eqref{eq:2sph-duality-sufficient}.
\end{proof}

The proofs that \eqref{eq:sum-jk-desired-bound}  implies \eqref{eq:L32-bound} and that \eqref{eq:sum-jkl-desired-bound} implies \eqref{eq:L43-bound} are similar to the proof above. In place of Cauchy--Schwarz, we apply H\"older to get
\begin{align*}
\| \sum_{i} a_i 1_{S_i^\delta} \|_3^3
&\leq
\sum_{i,j,k} a_i^3 |S_i^\delta \cap S_j^\delta \cap S_k^\delta|
\\
\| \sum_{i} a_i 1_{S_i^\delta} \|_4^4
&\leq
\sum_{i,j,k,\ell} a_i^4 |S_i^\delta \cap S_j^\delta \cap S_k^\delta \cap S_\ell^\delta|
.
\end{align*}

We prove \eqref{eq:sum-j-desired-bound}, \eqref{eq:sum-jk-desired-bound}, \eqref{eq:sum-jkl-desired-bound} in the next three subsections.

\subsection{Two spheres}\label{sec:L2unitsphere}

In our proof of \Cref{lemma:geometric-estimates-intersections-234}, we frequently use the estimate below. In applications of this estimate, $x_j$ is the center of a unit sphere containing $\omega_j$.

\begin{lemma}[Weighted counting estimate]
\label{lemma:basic-sum-estimate}
Let $n \geq 2$. Let $\{\omega_j\} \subset \R^n$ be a $\delta$-separated set. For each $j$, let $x_j \in \R^n$ be a point satisfying $|\omega_j - x_j| = 1$. Then for any $a \in \R^n$ and any $0 \leq P \leq Q\leq 1$,
\begin{align}
\label{eq:lemma:basic-sum-estimate}
\sum_{\substack{
    j
    \\ 
    P \leq |a-x_j| \leq Q
}} 
(|a-x_j|+\delta)^\alpha
\lesssim_{\alpha, n}
\begin{cases}
\delta^{-n} Q^{\alpha+1} &\text{if } \alpha > -1
\\
\delta^{-n} \log\frac{Q}{P+\delta} &\text{if } \alpha = -1
\\
\delta^{-n} (P+\delta)^{\alpha+1} &\text{if } \alpha < -1
\end{cases}
\end{align}
\end{lemma}

\begin{proof}
First we note the following: if $\delta \leq \rho \leq 1$ and  $|a-x_j| \lesssim \rho$, then $\omega_j$ must be in the $O(\rho)$-neighborhood of the unit sphere $\S^{n-1}+a$. Since $\{\omega_j\}$ is $\delta$-separated,
\begin{align} \label{eq:lemma:basic-sum-estimate:counting}
    \#\{ j : |a-x_j| \leq \rho \}
    \lesssim
    \delta^{-n} \rho.
\end{align}
Thus,
\begin{align}
\sum_{\substack{
j
\\ 
    P \leq |a-x_j| \leq Q
}} 
(|a-x_j|+\delta)^\alpha
&\lesssim
\sum_{\substack{
    \ell \in \Z
    \\ 
    P+\delta \lesssim 2^\ell \lesssim Q
}} 
\sum_{\substack{
    j
    \\ 
    |a-x_j|+\delta \approx 2^\ell
}} 
2^{\ell \alpha}
\lesssim
\delta^{-n}
\sum_{\substack{
    \ell \in \Z
    \\ 
    P+\delta \lesssim 2^\ell \lesssim Q
}} 
2^{\ell(\alpha+1)}
\end{align}
and the result follows.
\end{proof}

With \Cref{lemma:basic-sum-estimate} and an elementary estimate on the intersection of two spheres (\Cref{lem:intersection}), we can easily deduce \eqref{eq:sum-j-desired-bound}.

\begin{proof}[Proof of \eqref{eq:sum-j-desired-bound}]
Let $x_j$ denote the center of $S_j$, so that $|x_j - \omega_j| = 1$. 
By choosing $c_{diam} \leq \frac{1}{2}$ (say), we can apply the volume bound for the intersection from \Cref{lem:intersection}. This gives
\begin{align*}
\sum_{j}  |S_i^\delta \cap S_j^\delta| 
&\labelrel\les{eq:intersection-2-spheres-vary-special}
\sum_{j} \frac{\delta^2}{|x_i - x_j| + \delta}
\labelrel\les{eq:lemma:basic-sum-estimate}  
\delta^2 \delta^{-n} \log (1/\delta),
\end{align*}
which completes the proof.
\end{proof}

\begin{remark}
A Fourier analytic proof of the $L^2$ bound \eqref{eq:NM-L2} is given in \Cref{thm:NTfulldim}.
\end{remark}

\subsection{Three spheres}
\label{subsection:L32}

Before proving \eqref{eq:sum-jk-desired-bound}, we provide some informal remarks and motivation. To prove \eqref{eq:sum-jk-desired-bound}, we will use estimates on the intersection of three unit spheres, \Cref{lem:three_spheres}.
In the proof in the previous section, we avoided external tangencies between two spheres by making $c_{diam}$ small. However, for three spheres in $\R^3$, there is a ``tangential'' configuration that cannot be avoided, even if we make $c_{diam}$ small: Roughly speaking, if we have three well-separated points $a,b,c \in \R^3$ that lie on a unit circle, then $S^\delta(a) \cap S^\delta(b) \cap S^\delta(c)$ is approximately a ``tube'' of dimensions $\delta \times \delta \times \delta^{1/2}$. 

On the other hand, if $a,b,c$ lie on a circle of radius bounded away from $1$, then the three spheres intersect transversely and the volume is $O(\delta^3)$. In \Cref{lem:three_spheres}, there are different estimates depending on this circumradius.

If we choose all our unit spheres $\{S_i\}_{i \in I}$ so that their centers lie on a fixed unit circle, then the typical triple intersection is roughly $\delta^{5/2}$. (This is related to the ``tube'' example in the proof of \Cref{prop:NDeltaLowerBd}.) The bound \eqref{eq:sum-jk-desired-bound} asserts that this is essentially the worst configuration.

\begin{proof}[Proof of \eqref{eq:sum-jk-desired-bound}]
We adopt the following notational conventions in the upcoming calculations:  \begin{itemize}
\item We let $x_j$ denote the center of $S_j$.
\item We let $M_{ijk}, m_{ijk}, R_{ijk}$ denote the longest side length, shortest side length, and circumradius, respectively, of the triangle with vertices $x_i,x_j,x_k$. 
\end{itemize}

By choosing $c_{diam}$ smaller than the constant $c_2$ in \Cref{lem:three_spheres}, we can apply \Cref{lem:three_spheres} to every term in the left-hand side of \eqref{eq:sum-jk-desired-bound}. We split the sum over $j,k$ into three parts:
\begin{align}
\sum_{\substack{
j,k
}}
|S_i^\delta \cap S_j^\delta \cap S_k^\delta|
&\lesssim
\sum_{\substack{
j,k \\ |x_i-x_k| \leq |x_i-x_j|
}}
|S_i^\delta \cap S_j^\delta \cap S_k^\delta|
\\
&=
\sum_{\substack{
j,k \\ |x_i-x_k| \leq |x_i-x_j| \\ m_{ijk} \lesssim \sqrt{\delta}
}}  
+
\sum_{\substack{
j,k \\ |x_i-x_k| \leq |x_i-x_j| \\ m_{ijk} \gtrsim \sqrt{\delta} \\ R_{ijk} \leq \frac{1}{2}
}}
+
\sum_{\substack{
j,k \\ |x_i-x_k| \leq |x_i-x_j| \\ m_{ijk} \gtrsim \sqrt{\delta} \\ \frac{1}{2} \leq R_{ijk} \leq 2
}}  
\end{align}
(By \eqref{eq:empty-intersection}, we do not need to consider triples with $R_{ijk} \geq 2$.)  We now bound each of the three terms by the right-hand side of \eqref{eq:sum-jk-desired-bound}.
Note that the condition $|x_i-x_k| \leq |x_i-x_j|$ implies
\begin{align}
\label{eq:M-eq-ab}
M_{ijk} &\approx |x_i-x_j|
\\
\label{eq:m-eq-ac-or-bc}
m_{ijk} &= \min(|x_i-x_k|,|x_j-x_k|).
\end{align}

For the first term, we reduce to the case of two spheres. \begin{align*}
\sum_{\substack{
j,k \\ |x_i-x_k| \leq |x_i-x_j| \\ m_{ijk} \lesssim \sqrt{\delta}
}}  
&\labelrel\leq{eq:m-eq-ac-or-bc}
\sum_{\substack{
j
}}  
|S_i^\delta \cap S_j^\delta|
\sum_{\substack{
k \\ |x_i-x_k| \lesssim \sqrt{\delta} \text{ or } |x_j-x_k| \lesssim \sqrt{\delta} 
}}  
1
\\
&\labelrel\lesssim{eq:lemma:basic-sum-estimate}
\delta^{-n}\delta^{1/2}
\sum_j
|S_i^\delta \cap S_j^\delta|
\\
&\labelrel\lesssim{eq:sum-j-desired-bound}
\delta^{-n}\delta^{1/2} \delta^{-n} \delta^{2} \log (1/\delta)
\end{align*}

For the second term, we use the volume bound for the intersection of three unit spheres when the circumradius is bounded away from 1. 
This gives
\begingroup\allowdisplaybreaks
\begin{align*}
\sum_{\substack{
j,k \\ |x_i-x_k| \leq |x_i-x_j| \\ m_{ijk} \gtrsim \sqrt{\delta} \\ R_{ijk} \leq \frac{1}{2}
}}
&\labelrel\lesssim{eq:circumradius-small-estimate}
\delta^3
\sum_{\substack{
j,k \\ |x_i-x_k| \leq |x_i-x_j| \\ m_{ijk} \gtrsim \sqrt{\delta}
}} 
\frac{1}{M_{ijk}^2 m_{ijk}}
\\
&\labelrel\lesssim{eq:M-eq-ab}
\delta^3
\sum_{\substack{
j \\ |x_i-x_j| \gtrsim \sqrt{\delta}
}} 
\frac{1}{|x_i-x_j|^2}
\sum_{\substack{
k \\ |x_i-x_k| \leq |x_i-x_j| \\ m_{ijk} \gtrsim \sqrt{\delta}
}} 
\frac{1}{m_{ijk}}
\\
&\labelrel\leq{eq:m-eq-ac-or-bc} 
\delta^3
\sum_{\substack{
j \\ |x_i-x_j| \gtrsim \sqrt{\delta}
}} 
\frac{1}{|x_i-x_j|^2}
\left(
\sum_{\substack{
k \\ \sqrt{\delta} \lesssim |x_i-x_k| \leq |x_i-x_j|
}} 
\hspace{-.5cm}
\frac{1}{|x_i-x_k|}
+
\sum_{\substack{
k \\ \sqrt{\delta} \lesssim |x_j-x_k| \lesssim |x_i-x_j|
}} 
\hspace{-.5cm}
\frac{1}{|x_j-x_k|}
\right)
\\
&\labelrel\lesssim{eq:lemma:basic-sum-estimate}
\delta^{3} \delta^{-n}\log (1/\delta)
\sum_{\substack{
j \\ |x_i-x_j| \gtrsim \sqrt{\delta}
}} 
\frac{1}{|x_i-x_j|^2}
\\
&\labelrel\lesssim{eq:lemma:basic-sum-estimate}
\delta^{3} \delta^{-n}\log (1/\delta) \delta^{-n} \delta^{-1/2}
\end{align*}
\endgroup

For the third term, we use the remaining volume bound for the intersection of three unit spheres.  This gives
\begingroup\allowdisplaybreaks
\begin{align*}
\sum_{\substack{
j,k \\ |x_i-x_k| \leq |x_i-x_j| \\ m_{ijk} \gtrsim \sqrt{\delta} \\ \frac{1}{2} \leq R_{ijk} \leq 2
}}  
&\labelrel\lesssim{eq:circumradius-near-1-estimate}
\delta^{5/2}
\sum_{\substack{
j,k  \\ |x_i-x_k| \leq |x_i-x_j| \\ m_{ijk} \gtrsim \sqrt{\delta} \\ \frac{1}{2} \leq R_{ijk} \leq 2
}} 
\frac{1}{M_{ijk}^{3/2} m_{ijk}^{1/2}} 
\\
&\labelrel\lesssim{eq:M-eq-ab}
\delta^{5/2}
\sum_{\substack{
j \\ |x_i-x_j| \gtrsim \sqrt{\delta}
}} 
\frac{1}{|x_i-x_j|^{3/2}}
\sum_{\substack{
k  \\ |x_i-x_k| \leq |x_i-x_j| \\ m_{ijk} \gtrsim \sqrt{\delta}
}} 
\frac{1}{m_{ijk}^{1/2}}
\\
&\labelrel\leq{eq:m-eq-ac-or-bc}
\delta^{5/2}
\sum_{\substack{
j \\ |x_i-x_j| \gtrsim \sqrt{\delta}
}} 
\frac{1}{|x_i-x_j|^{3/2}}
\left(
\sum_{\substack{
k \\ \sqrt{\delta} \lesssim |x_i-x_k| \leq |x_i-x_j|
}} 
\hspace{-.7cm}
\frac{1}{|x_i-x_k|^{1/2}}
+
\sum_{\substack{
k \\ \sqrt{\delta} \lesssim |x_j-x_k| \lesssim |x_i-x_j|
}} 
\hspace{-.7cm}
\frac{1}{|x_j-x_k|^{1/2}}
\right)
\\
&\labelrel\lesssim{eq:lemma:basic-sum-estimate}
\delta^{5/2}
\delta^{-n}
\sum_{\substack{
j \\ |x_i-x_j| \gtrsim \sqrt{\delta}
}} 
\frac{1}{|x_i-x_j|^{3/2}}|x_i-x_j|^{1/2}
\\
&\labelrel\lesssim{eq:lemma:basic-sum-estimate}
\delta^{5/2}
\delta^{-n}
\delta^{-n}
\log(1/\delta)
\end{align*}
\endgroup
This completes the proof of \eqref{eq:sum-jk-desired-bound}.
\end{proof}

\subsection{Four spheres}

We again begin with some informal remarks. Unlike in the proofs of \eqref{eq:sum-j-desired-bound} and \eqref{eq:sum-jk-desired-bound}, here we do not need an analogue of \Cref{lem:intersection} and \Cref{lem:three_spheres} for four unit spheres. If four unit spheres in $\R^4$ are centered on $S^2 \times \{0\} \subset \R^4$, then the intersection of the $\delta$-neighborhoods has dimensions approximately $\delta \times \delta \times \delta \times \delta^{1/2}$. However, it turns out there is an arrangement with larger intersections.

If we place unit spheres centered on the circle $\frac{1}{\sqrt{2}}S^1 \times \{(0,0)\} \subset \R^4$, then all of these spheres contain the circle $\{(0,0)\} \times \frac{1}{\sqrt{2}}S^1$. This implies that the intersection of the $\delta$-neighborhoods of four spheres has volume like $\delta^3$. (This is related to the ``radius $1/\sqrt{2}$'' example in the proof of \Cref{prop:NDeltaLowerBd}.) 

Note that in this configuration,
\[
S_i \cap S_j \cap S_k = S_i \cap S_j \cap S_k \cap S_\ell = \{(0,0)\} \times \frac{1}{\sqrt{2}}S^1,
\]
so the fourth sphere does not contribute to the intersection. This suggests that the trivial estimate
\begin{align}
\label{eq:bound-4-by-3}
|S_i^\delta \cap S_j^\delta \cap S_k^\delta \cap S_\ell^\delta|
\leq
|S_i^\delta \cap S_j^\delta \cap S_k^\delta|
\end{align} 
might not lose too much in some situations. Indeed, our proof begins by relabeling $j, k, \ell$ and applying \eqref{eq:bound-4-by-3}, and this is why we do not need an analogue of \Cref{lem:three_spheres} for four spheres.

\begin{proof}[Proof of \eqref{eq:sum-jkl-desired-bound}]
Let $m_{ijk\ell}$ denote the length of the shortest edge of the tetrahedron formed by $x_i, x_j, x_k, x_\ell$. By symmetry in $j,k,\ell$, we can assume the shortest side length is $|x_i - x_\ell|$ or $|x_j - x_\ell|$. Under this assumption, we use \eqref{eq:bound-4-by-3}, giving us the estimate
\begin{align}
\sum_{\substack{
j,k,\ell 
}}
|S_i^\delta \cap S_j^\delta \cap S_k^\delta \cap S_\ell^\delta|
\lesssim
\sum_{\substack{
j,k,\ell \\
m_{ijk\ell} = \min(|x_i - x_\ell|,|x_j - x_\ell|)
}}
|S_i^\delta \cap S_j^\delta \cap S_k^\delta|
\end{align}
We further split the sum into two terms, depending on whether $m_{ijk\ell} \lesssim \sqrt{\delta}$ or $m_{ijk\ell} \gtrsim \sqrt{\delta}$, and we now bound each of these terms separately.

For the $m_{ijk\ell} \lesssim \sqrt{\delta}$ term, we use \eqref{eq:sum-jk-desired-bound}.
\begin{align*}
\sum_{\substack{
j,k,\ell \\
m_{ijk\ell} = \min(|x_i - x_\ell|,|x_j - x_\ell|) \\
m_{ijk\ell} \lesssim \sqrt{\delta}
}}
&\leq
\sum_{\substack{
j,k
}}
|S_i^\delta \cap S_j^\delta \cap S_k^\delta|
\sum_{\substack{
\ell \\
\min(|x_i - x_\ell|,|x_j - x_\ell|) \lesssim \sqrt{\delta}
}}
1
\\
&\labelrel\lesssim{eq:lemma:basic-sum-estimate}
\delta^{-n} \delta^{1/2}
\sum_{\substack{
j,k
}}
|S_i^\delta \cap S_j^\delta \cap S_k^\delta|
\\
&\labelrel\lesssim{eq:sum-jk-desired-bound}
\delta^{-n} \delta^{1/2} \delta^{5/2-2n} \log(1/\delta).
\end{align*}

For the $m_{ijk\ell} \gtrsim \sqrt{\delta}$ term, we use the $n \geq 4$ case of \Cref{lem:three_spheres}, the fact that $m_{ijk} \geq m_{ijk\ell}$, and arguments similar to those of \Cref{subsection:L32}.
\begingroup\allowdisplaybreaks
\begin{align*}
\sum_{\substack{
j,k,\ell \\
m_{ijk\ell} = \min(|x_i - x_\ell|,|x_j - x_\ell|) \\
m_{ijk\ell} \gtrsim \sqrt{\delta}
}}
&\labelrel\lesssim{eq:circumradius-small-estimate}
\delta^3
\sum_{\substack{
j,k \\ m_{ijk} \gtrsim \sqrt{\delta}
}} 
\frac{1}{M_{ijk}^2 m_{ijk}}
\sum_{\substack{
\ell \\
\min(|x_i - x_\ell|,|x_j - x_\ell|) \leq m_{ijk}
}}
1
\\
&\labelrel\lesssim{eq:lemma:basic-sum-estimate}
\delta^3
\delta^{-n}
\sum_{\substack{
j,k \\ m_{ijk} \gtrsim \sqrt{\delta}
}} 
\frac{1}{M_{ijk}^2}
\\
&\lesssim
\delta^3
\delta^{-n}
\sum_{\substack{
j,k \\ |x_i - x_k| \leq |x_i - x_j| \\ m_{ijk} \gtrsim \sqrt{\delta}
}} 
\frac{1}{M_{ijk}^2}
\\
&\labelrel\lesssim{eq:M-eq-ab}
\delta^3
\delta^{-n}
\sum_{\substack{
j \\  |x_i - x_j| \gtrsim \sqrt{\delta}
}} 
\frac{1}{|x_i - x_j|^2}
\sum_{\substack{
k \\ \sqrt{\delta} \lesssim |x_i - x_k| \leq |x_i - x_j|
}} 
1
\\
&\labelrel\lesssim{eq:lemma:basic-sum-estimate}
\delta^3
\delta^{-n}
\delta^{-n}
\sum_{\substack{
j \\  |x_i - x_j| \gtrsim \sqrt{\delta}
}} 
\frac{1}{|x_i - x_j|}
\\
&\labelrel\lesssim{eq:lemma:basic-sum-estimate}
\delta^3
\delta^{-n}
\delta^{-n}
\delta^{-n}
\log(1/\delta)
\end{align*}
\endgroup

This completes the proof of \eqref{eq:sum-jkl-desired-bound}.
\end{proof}

\begin{remark}\label{remark:NT-same-bounds-as-NM}
Let $T \subset \R^n$. For $\delta \in (0,1/2)$, we define  
\[ \NT f(x) = \sup_{u\in T}  \left|\frac{1}{|S^{\delta}(0)|} \int_{S^\delta(0)} f(x+u+y)dy \right|. \]
Note that $\NT = \NM$ for $T=\S^{n-1}$.

If $T$ is any compact set with finite $(n-1)$-dimensional upper Minkowski content (recall \eqref{eqn:coverassume}), then \Cref{theorem:NM-upper-bounds}\ref{theorem:NM-upper-bounds-item:upper-bounds} also holds if we replace $\NM$ with $\NT$. To see this, we only need to make the following changes to \Cref{lemma:basic-sum-estimate}: 
\begin{enumerate}
    \item In the statement of the lemma, we replace $|\omega_j - x_j| = 1$ with $x_j - \omega_j \in T$. 
    \item In the proof, if $\delta \leq \rho \leq 1$ and $|a-x_j| \lesssim \rho$, then $\omega_j-a$ must be in the $O(\rho)$-neighborhood of $-T$. Since $\{\omega_j\}$ is $\delta$-separated, the covering condition on $T$ implies \eqref{eq:lemma:basic-sum-estimate:counting}. The rest of the proof is unchanged.
\end{enumerate}
\end{remark}
 
\section{The Nikodym maximal function $\NS$ associated with spheres of varying radii}\label{sec:NS}
In this section, we give a proof of \Cref{theorem:NS-upper-bounds}\ref{theorem:NS-upper-bounds-item:upper-bounds}. In fact, we will prove a slightly more general result. For a compact set $T\subset \R^n$, consider the maximal function 
\[ \NS_T f(x) = \sup_{u\in T} \sup_{1\leq t\leq 2} \frac{1}{|S^{\delta}(0)|} \left|\int_{S^\delta(0)} f(x+t(u+y))dy \right|.\]
We recall that $\NS= \NS_T$ for $T=\S^{n-1}$. 

\begin{theorem}\label{thm:NS3}
Let $n\geq 3$. Suppose that $T\subset \R^n$ is a compact set with finite $(n-1)$-dimensional upper Minkowski content (see \eqref{eqn:coverassume}). Then
\begin{align*}
 \|\NS_T f \|_{L^2(\R^n)} &\les (\log \delta^{-1})^{1/2} \|f \|_{L^2(\R^n)}.
\end{align*}
\end{theorem}

\begin{theorem}\label{thm:NS2}
Let $n=2$.  Suppose that $T\subset \R^n$ is a compact set with finite $(n-1)$-dimensional upper Minkowski content (see \eqref{eqn:coverassume}). Then for any $\epsilon>0$,
\[ \|\NS_T f \|_{L^3(\R^2)} \les_\epsilon \delta^{-\epsilon} \|f \|_{L^3(\R^2)}. \]
\end{theorem}

Note that \Cref{theorem:NS-upper-bounds}\ref{theorem:NS-upper-bounds-item:upper-bounds} is a consequence of interpolations between bounds from Theorems \ref{thm:NS3} and \ref{thm:NS2} and trivial $L^1$ and $L^\infty$ bounds.

\subsection{Proof of Theorem \ref{thm:NS3}}
We first note that the maximal function $\NS_T$ is local in the sense that if $f$ is a function supported on a ball $B$ of radius $1$, then $\NS_T f$ is supported on a ball of radius $O(1)$ sharing the same center with $B$. Therefore, it is sufficient to bound $\NS_T f$ on a ball of radius 1. 

We first generalize the counting result \eqref{eq:lemma:basic-sum-estimate:counting} to any compact $T$ satisfying \eqref{eqn:coverassume} for $s=n-1$. In what follows, we identify a sphere $S(x,t) = \{ y\in \R^n: |y-x|=t \}$ with the point $(x,t)\in \R^{n+1}$. Accordingly, given spheres $S_i$ and $S_j$, $|S_i - S_j|$ denotes the usual distance between two points in $\R^{n+1}$. 

\begin{lemma}\label{lem:count} Let $0\leq s\leq n$. Suppose that $T\subset \R^n$ is a compact set such that $N(T,\delta) \les \delta^{-s}$  for all $\delta\in (0,1/2)$. Let $\{\omega_i\}$ be a $\delta$-separated set of points of $\R^n$. For each $i$, we let $S_i=S_i(x_i,t_i)$ be a sphere, where $t_i \in [1,2]$ and $x_i = w_i+ t_iu_i$ for some $u_i\in T$. Then 
\[  \#\{i : S_i \in B_\rho \}
\les \delta^{-n} \rho^{n-s}  \]
for any ball $B_\rho \subset \R^{n+1}$ of radius $\delta\leq \rho \leq 1$.
\end{lemma}
\begin{proof}
Let $(x,t)\in \R^{n+1}$ be the center of $B_\rho$. Since $t_i\in [1,2]$, without loss of generality, we may assume that $t\in [1,2]$. We claim that $\omega_i$ belongs to $\nbd{x-tT}{O(\rho)}$ if $S_i \in B_\rho$. Since the volume of $\nbd{x-tT}{O(\rho)}$ is $O(\rho^{n-s})$, it may contain $O(\delta^{-n}\rho^{n-s})$ $\delta$-separated set of points $\omega_i$, which completes the proof. 

To verify the claim, it suffices to show that $|\omega_i - (x - t u_i) | = O(\rho)$. By the triangle inequality,
\[ |\omega_i - (x - t u_i) |  \leq |x_i - x| + |t_i-t| |u_i|.\]
The claim follows from the assumption $S_i \in B_\rho$ and the compactness of $T$.
\end{proof}

We are ready to prove Theorem \ref{thm:NS3}. By the same duality argument for $\NM$, it suffices to prove the following; Let $\{\omega_i\}$ be a $\delta$-separated subset of $\R^n$. Then
\begin{equation}\label{eqn:NS3:dual}
    \delta^{n-1} \sup \| \sum_{i} a_i 1_{S_i^{\delta}} \|_{L^{2}(\R^n)} \lesssim \log \delta^{-1}
\end{equation}
where the supremum is taken over all choices of spheres $S_i=S_i(x_i,t_i)$, where $x_i = w_i+t_i u_i$ for some $t_i\in [1,2]$ and $u_i\in T$, and $a_i\geq0$ such that $\delta^n \sum_{i} a_i^{2} = 1$. We know that by \Cref{lem:count}, any such collection of spheres satisfies the bound
\begin{equation} \label{eq:NS2:count}
    \#\{i : S_i \in B_\rho \} \leq A\rho
\end{equation}
for $A\sim \delta^{-n}$, for any ball $B_\rho \subset \R^{n+1}$ of radius $\delta\leq \rho \leq 1$. Hence, \eqref{eqn:NS3:dual} is a consequence of the following. 

\begin{proposition} \label{prop:NS3dual} Let $n\geq 3$, $\delta \in (0,1/2)$, and $A>0$. Let $\{ S_i \}$ be a collection of spheres of radius comparable to 1 satisfying \eqref{eq:NS2:count} for any ball $B_\rho \subset \R^{n+1}$ of radius $\delta\leq \rho \leq 1$. Then  
\[ \| \sum_{i} a_{i} 1_{S_i^\delta} \|_{L^{2}(\R^n)} \les \delta \log \delta^{-1} A^{1/2}  \left(\sum_{i} a_i^{2}\right)^{1/2} \]
for any $a_i \in [0,\infty)$.
\end{proposition} 
\begin{proof}
The proof is a minor modification of the proof of the $L^2$ bound \eqref{eq:NM-L2} for $\NM$. Therefore, we only indicate necessary modifications. In Section \ref{sec:L2unitsphere}, we replace $|x_i- x_j|$ by $|S_i - S_j|$ and use \Cref{lemma:basic-sum-estimate} with $|a-x_j|$ replaced by $|S_i- S_j|$.
\end{proof}

\Cref{prop:NS3dual} cannot be extended to the $n=2$ case. This is due to the fact that the volume bound $|S_i^\delta \cap S_j^\delta|  \les
\frac{\delta^2}{|S_i-S_j|+ \delta}$ used in the proof fails to hold for $n=2$ in general; indeed, $|S_i^\delta \cap S_j^\delta|\sim \delta^{3/2}$ for two internally tangent circles $S_i$ and $S_j$ such that $|S_i-S_j| \sim 1$.

\subsection{Proof of Theorem \ref{thm:NS2}}
By the duality argument in \Cref{lem:duality-unit-radius}, it suffices to prove that, for $p'=3/2$, 
\begin{equation}\label{eqn:NS2:dual}
    \delta \sup \| \sum_{i} a_i 1_{C_i^{\delta}} \|_{L^{p'}(\R^2)} \less 1,
\end{equation}
where the supremum is taken over all choices of circles $C_i=C_i(x_i,t_i)$, where $x_i = w_i+t_i u_i$ for a $\delta$-separated set of points $\{w_i\}$,  $t_i\in [1,2]$ and $u_i\in T$, and $a_i\geq0$ such that $\delta^2 \sum_{i} a_i^{p'} = 1$. As in the proof of Theorem \ref{thm:NS3}, \eqref{eqn:NS2:dual} is a consequence of the following.

\begin{proposition}\label{prop:M2} Let $n=2$, $\delta \in (0,1/2)$, $0 < A < \delta^{-O(1)}$. Let $\{ C_i \}_{i\in I}$ be a collection of circles of radius comparable to 1 satisfying 
\begin{align}
\label{eq:M2-nonconcentration}
\# \{i \in I : C_i \in B \} \leq A r/\delta    
\end{align}
for any ball $B \subset \R^3$ of radius $\delta\leq r \leq 1$.
Then  
\[ 
\| \sum_{i \in I} 1_{C_i^\delta} \|_{L^{3/2}(\R^2)} \leq C \delta^{-\epsilon} A^{1/3}   (\delta \# I )^{2/3}, 
\]
where $C$ depends only on $\epsilon$ and on the $O(1)$ constant appearing in the upper bound of $A$.
\end{proposition} 

To be more precise, for \eqref{eqn:NS2:dual}, we use a strong-type inequality
\[
   \| \sum_{i \in I} a_i 1_{C_i^\delta} \|_{L^{3/2}(\R^2)} \less A^{1/3} \left(\delta\sum_{i \in I} a_i^{3/2} \right)^{2/3}  
\]
for any $a_i\geq 0$. This inequality follows from the restricted strong-type inequality in \Cref{prop:M2} by an interpolation argument (see \Cref{prop:duality} for details). Applying it with $A \sim \delta^{-1}$ gives \eqref{eqn:NS2:dual}.

\begin{proof}
\Cref{prop:M2} is a consequence of \cite{Pramanik-Yang-Zahl}. We provide some details. We may assume that $\{ C_i\}_{i\in I} \subset B$ for a fixed ball $B\subset \R^3$ of radius $\sim 1$ as we may assume that the centers of the circles $\{ C_i\}_{i\in I}$ lie in a set of diameter $\sim 1$ (cf. \eqref{eq:bounded-overlap}).  First note that if $\# I \leq A$, then \eqref{eq:M2-nonconcentration} also holds with with $(\#I) (r/\delta)$ on the right-hand side. Thus, it suffices to prove this theorem when $\# I \geq A$.

Fix $\epsilon > 0$. We define a random subset $J \subset I$ by including each $i \in I$ independently with probability $p = \delta^{\epsilon}/A$.
First, by the Chernoff bound and $\#I \geq A$,
\begin{align*}
\P\left[ \# J \geq \delta^{-2\epsilon} p \# I \right] 
\leq
\left(\frac{e p \#I}{\delta^{-2\epsilon}p \# I}\right)^{\delta^{-2\epsilon}p \# I}
\leq
\exp(-\delta^{-\epsilon})
.
\end{align*}
(In the second inequality, we assume $\delta$ is sufficiently small, depending on $\epsilon$.)
Similarly, by Chernoff and \eqref{eq:M2-nonconcentration},
\begin{align*}
\P\left[ \# \{ i \in J : C_i \in B\} \geq \delta^{-\epsilon}r/\delta\right] 
&\leq 
\left(\frac{e p \#\{ i \in I : C_i \in B\}}{\delta^{-\epsilon}r/\delta}\right)^{\delta^{-\epsilon}r/\delta}
\leq 
\exp(-\delta^{-\epsilon})
\end{align*}
for any ball $B \subset \R^3$ of radius $r \geq \delta$. (In the second inequality, we assume $\delta$ is sufficiently small, depending on $\epsilon$.)

As noted in \cite[Section 2]{Pramanik-Yang-Zahl} (see the part after equation (2.14)), with probability $1-O(\delta^{-3}\exp(-\delta^\epsilon))$, both of the following occur.
\begin{gather}
\label{eq:M2-J-not-too-big}
\# J \lesssim \delta^{-2\epsilon} p \#I
\\
\label{eq:M2-J-nonconcentration}
\# \{i \in J : C_i \in B \} \lesssim \delta^{-\epsilon} r/\delta  
\quad \text{for any ball $B \subset \R^3$ of radius $r \geq \delta$},
\end{gather}

Next, by Jensen's inequality,
\begin{align*}
\E \| \sum_{i \in J} 1_{C_i^\delta} \|_{3/2}^{3/2}
\geq
\| \E \sum_{i \in J} 1_{C_i^\delta} \|_{3/2}^{3/2} 
=
p^{3/2} \| \sum_{i \in I} 1_{C_i^\delta} \|_{3/2}^{3/2}
\end{align*}
so the event
\begin{align}
\label{eq:M2-J-norm}
\| \sum_{i \in J} 1_{C_i^\delta} \|_{3/2}^{3/2}
 \geq \frac{1}{2} p^{3/2} \| \sum_{i \in I} 1_{C_i^\delta} \|_{3/2}^{3/2}
\end{align}
occurs with probability at least $\frac{1}{2} p^{3/2} = \frac{1}{2} (\frac{\delta^\epsilon}{A})^{3/2}$.

Since $A \leq \delta^{-O(1)}$, we have $\delta^{-3}\exp(-\delta^\epsilon) \ll (\frac{\delta^\epsilon}{A})^{3/2}$. Thus, there exists a $J \subset I$ such \eqref{eq:M2-J-not-too-big}, \eqref{eq:M2-J-nonconcentration}, and \eqref{eq:M2-J-norm} all hold, so we can apply \cite[Theorem 1.7]{Pramanik-Yang-Zahl} (see the second remark after the theorem) to obtain 
\begin{align*}
p^{3/2} \| \sum_{i \in I} 1_{C_i^\delta} \|_{3/2}^{3/2}
\labelrel\lesssim{eq:M2-J-norm}
\| \sum_{i \in J} 1_{C_i^\delta} \|_{3/2}^{3/2}
\lesssim
\delta^{-O(\epsilon)} \delta \# J
\labelrel\lesssim{eq:M2-J-not-too-big}
\delta^{-O(\epsilon)} p \delta \# I,
\end{align*}
which completes the proof.
\end{proof}

\subsection{Maximal functions associated with spheres} \label{sec:maxspheres}
We discuss a class of maximal functions associated with spheres which include $\NS$ as a special case. Let $\Omega \subset \R^m$ be a set of measure $|\Omega| \les 1$ for some $m\geq 1$.  Suppose that for each  $\omega \in \Omega$, we are given a collection of spheres $
\A(\omega) \subset \R^n\times [1,2]$. Then for $f : \mathbb{R}^n \to \mathbb{R}$, we define the maximal function $M^\delta_{\A} f : \Omega \to \mathbb{R}$ by 
\begin{align}
M^\delta_{\A} f(\omega) = \sup_{S \in \A(\omega)} \frac{1}{|S^{\delta}|} \left| \int_{S^{\delta}} f \right|,
\end{align}
where $S^\delta$ denotes the $\delta$-neighborhood of the sphere $S$ in $\R^n$.

We give some concrete examples. Let $\Omega \subset \R^n$ be a ball and $T\subset \R^n$ be a compact set with finite $(n-1)$-dimensional upper Minkowski content (see \eqref{eqn:coverassume}). If we take 
\[ \A(\omega) = \{ (\omega+tu, t) \in \R^{n+1} : u\in T , t\in [1,2] \}, \] then $M^\delta_{\A} = \NS_T$. When $\Omega = [1,2]$ and 
\[ \A(\omega) = \R^n \times \{ \omega \}, \] then $M^\delta_{\A}$ is the maximal function $W^\delta$ (see \eqref{eqn:Wdelta}).

We make the following assumptions on the collection $\{\A(\omega) \}_{\omega \in \Omega}$.
\begin{enumerate}[label={(MAX\arabic*)},leftmargin=\widthof{[MAX2]}+\labelsep]
\item \label{item:comparability} For any $\omega, \omega' \in \Omega$ with $|\omega - \omega'| \leq \delta$,  $M^\delta_{\A} f(\omega) \lesssim M^{O(\delta)}_{\A} f(\omega')$.

\item \label{item:nonconcentration} Let $\{\omega_i\} \subset \Omega$ be a $\delta$-separated set and $S_i\in \A(\omega_i)$ for each $i$. Then
 \[ \#\{i : S_i \in B_\rho \} \les \delta^{-m} \rho \]
for any ball $B_\rho\subset \R^{n+1}$ of radius $\rho \in [\delta,1]$.
\end{enumerate} 

Note that these assumptions are satisfied in both of the above examples. Therefore, the following result generalizes Theorems \ref{thm:NS3} and  \ref{thm:NS2} as well as results from \cite{Kol_Wol, Wolff_Circ}.

\begin{theorem}\label{thm:M3}
Let $\{\A(\omega) \}_{\omega \in \Omega}$ be a collection satisfying the assumptions \ref{item:comparability} and \ref{item:nonconcentration}. Then 
\[ \|M^\delta_{\A} f \|_{L^2(\Omega)} \les (\log \delta^{-1})^{1/2} \|f \|_{L^2(\R^n)} \]
for $n\geq 3$. For $n=2$, 
\[ \|M^\delta_{\A} f \|_{L^3(\Omega)} \less \|f \|_{L^3(\R^n)}. \]
\end{theorem}

We omit the proof since it is similar to the proof of Theorems \ref{thm:NS3} and \ref{thm:NS2}.

\section{The maximal functions $\MT$ and $\ST$}\label{sec:ST}
In this section, we prove \Cref{thm:MT}, \Cref{thm:ST} and \Cref{thm:avg_frac}. 

\subsection{Reductions using Littlewood-Paley theory} \label{sec:Littlewood-Paley}
We first start by discussing a standard reduction using Littlewood-Paley theory. 
Let $\psi$ be a smooth function such that $\ft{\psi}$ is a radial function  supported on $\{
\xi: 1/2\leq |\xi|\leq 2 \}$, $0\leq \ft{\psi} \leq 1$, and $\sum_{j\in \Z} \ft{\psi_j}(\xi) = 1$ for any $\xi\neq 0$, where $\ft{\psi_j}(\xi) = \ft{\psi}(2^{-j} \xi)$. We also let $\varphi_{k} = \sum_{j\leq k} \psi_j$, whose Fourier transform is a smooth bump function supported on $\{\xi: |\xi|\leq 2^{k+1} \}$. In addition, we fix a function $\tilde{\psi}$ whose Fourier transform is supported on $\{ \xi :1/4\leq |\xi| \leq 4 \}$ such that $\psi = \psi * \tilde{\psi}$. We set $\ft{\tilde{\psi}_j}(\xi) = \ft{\tilde{\psi}}(2^{-j} \xi)$ so that $\psi_j = \psi_j * \tilde{\psi}_j$.

Let $\sigma_t$ be the measure on $t\S^{n-1}$ defined by $\ft{\sigma_t}(\xi) = \ft{\sigma}(t\xi)$. For $t\sim 1$, an elementary computation shows that $\psi_j * \sigma_t$ is $O(2^j)$ and decays rapidly away from the $O(2^{-j})$ neighborhood of $t\S^{n-1}$. To be specific, we have for all $j\geq 1$ and $N>0$,
\begin{equation}\label{eqn:pointwise}
\begin{split}
   |\psi_j * \sigma_t (x)| &\les_N 2^j(1+2^j| |x|- t|)^{-N} \les_N 2^j (1+|x|)^{-N}.    \\
      |\varphi_0 * \sigma_t (x)| &\les_N (1+|x|)^{-N}.
\end{split}
\end{equation}

In the following, we write 
\[\avg f (x,t) := \int_{\S^{n-1}} f(x+ty)d\sigma(y) = f*\sigma_t (x).\]
\begin{proposition} \label{prop:reduce} 
Let $j\geq 1$. If
\begin{equation}\label{eqn:STreduction}
    \| \sup_{u\in T} \sup_{1\leq t\leq 2} |\avg(f*\psi_j)(\cdot + tu,t)| \|_{L^p(\R^n)} \les A  \| f \|_{L^p(\R^n)}
\end{equation}
holds for some $2\leq p< \infty$ and $A>0$, then 
\begin{equation}\label{eqn:STreduction2}
    \| \sup_{l\in \Z} \sup_{u\in T} \sup_{1\leq t\leq 2} |\avg(f*\psi_{j-l})(\cdot + 2^ltu,2^lt)| \|_{L^p(\R^n)} \les A \| f \|_{L^p(\R^n)}.
\end{equation}
Let $1<p<\infty$. If there exists $\epsilon>0$ such that 
\begin{equation}\label{eqn:STreduction2_decay}
    \| \sup_{l\in \Z} \sup_{u\in T} \sup_{1\leq t\leq 2} |\avg(f*\psi_{j-l})(\cdot + 2^ltu,2^lt)| \|_{L^p(\R^n)} \les 2^{-j \epsilon} \| f \|_{L^p(\R^n)}
\end{equation}
for all $j\geq 1$, then $\ST$ is bounded on $L^p(\R^n)$. 
\end{proposition}
\begin{proof}
The proof is standard, so we sketch the proof. Using a rescaled version of \eqref{eqn:STreduction}, we get
\begin{align*}
    &\| \sup_{l\in \Z} \sup_{u\in T} \sup_{1\leq t\leq 2} |\avg(f*\psi_{j-l})(\cdot + 2^ltu,2^lt)| \|_{L^p(\R^n)}^p  \\
    \leq& \sum_{l\in \Z} \| \sup_{u\in T} \sup_{1\leq t\leq 2} |\avg(f*\psi_{j-l})(\cdot + 2^ltu,2^lt)| \|_{L^p(\R^n)}^p \\
    \les& A^p \sum_{l\in \Z}  \| f*\psi_{j-l} \|_{L^p(\R^n)}^p \les A^p \| f \|_{L^p(\R^n)}^p.
\end{align*}
In the last step, we used the embedding $l^2 \hookrightarrow l^p$ and the Littlewood-Paley theory.

Next, we make the decomposition
\begin{equation}\label{eqn:LP}
\begin{split}
    \ST f(x) =& \sup_{l\in \Z} \sup_{u\in T} \sup_{1\leq t\leq 2} |\avg f(x + 2^ltu,2^lt)| \\
    \leq & \sum_{j\geq 1} \sup_{l\in \Z} \sup_{u\in T} \sup_{1\leq t\leq 2} |\avg(f*\psi_{j-l})(x + 2^ltu,2^lt)| \\
    &+  \sup_{l\in \Z} \sup_{u\in T} \sup_{1\leq t\leq 2} |\avg(f*\varphi_{-l})(x + 2^ltu,2^lt)|.
\end{split}
\end{equation} 

Assuming \eqref{eqn:STreduction2_decay}, we may bound the first term in \eqref{eqn:LP} by using the triangle inequality and the exponential decay in $j$. The second term in \eqref{eqn:LP} is bounded pointwise by the Hardy-Littlewood maximal function of $f$;
\begin{equation}\label{eqn:HardyLittlewood}
   \sup_{l\in \Z} \sup_{u\in T} \sup_{1\leq t\leq 2} |f* \varphi_{-l} * \sigma_{2^lt}(x+2^ltu)| \les M_{HL}f(x). 
\end{equation}
Indeed, a rescaling of \eqref{eqn:pointwise} implies the pointwise estimate 
\[ |\varphi_{-l} * \sigma_{2^lt} (x+2^ltu)| \les 2^{-ln}(1+|2^{-l}x+tu|)^{-N} \les 2^{-ln}(1+|2^{-l}x|)^{-N} \]
for $t\sim 1$ and $u\in T$ as $|tu| \les 1$. Therefore, $\ST$ is bounded on $L^p$.
\end{proof}

Here we record a weak-type estimate to be used later.
\begin{lemma} \label{lem:L1} For every $j\geq 1$,
\[ \| \sup_{l\in \Z} \sup_{u\in T} \sup_{1\leq t\leq 2} |\avg(f*\psi_{j-l})(\cdot + 2^ltu,2^lt)| \|_{L^{1,\infty}(\R^n)} \les 2^{j} \| f \|_{L^1(\R^n)}.\]
\end{lemma}
\begin{proof}
It follows from the pointwise estimate, cf. \eqref{eqn:HardyLittlewood},
\[  \sup_{l\in \Z} \sup_{u\in T} \sup_{1\leq t\leq 2} |f* \psi_{j-l}* \sigma_{2^lt}(x+2^ltu)| \les 2^j M_{HL}f(x) \]
and the fact that $M_{HL}$ is of weak type $(1,1)$.
\end{proof}

\subsection{Proof of \Cref{thm:MT}}\label{section:NT}
We prove \Cref{thm:MT} in this subsection. The approach here also provides a Fourier analytic proof of the $L^2$-bound \eqref{eq:NM-L2}; see \Cref{thm:NTfulldim} below. We note that this subsection is independent of the remainder of this section and will not be used for the proof of \Cref{thm:ST}.

We recall the decomposition of the measure $\sigma  = \sigma* \varphi_0 + \sum_{j\geq 1} \sigma * \psi_j$. We note that $\varphi_0$ and $\psi_1$ satisfy similar bounds, so Theorem \ref{thm:MT} is obtained by summing over the following frequency localized estimates \Cref{prop:AVG} and \Cref{prop:AVG-using-NT}.

\begin{proposition} \label{prop:AVG} Let $n\geq 2$ and $0 \leq s \leq n$.\footnote{Unlike many other bounds in our paper, this bound is valid for $0\leq s \leq n$.} Suppose that $T\subset \R^n$ is a compact set with finite $s$-dimensional upper Minkowski content (see \eqref{eqn:coverassume}). Then for $1\leq p\leq 2$, 
\[ \| \sup_{u\in T} |f*\psi_j*\sigma(\cdot+u)| \|_{p}  \les  \min( 2^{-j( n-1-\frac{n-1+s}{p} )}, 2^{-j( n-s-\frac{n-s+1}{p} )}) \| f \|_p. \]
\end{proposition}

\begin{proposition} \label{prop:AVG-using-NT} 

Let $n\geq 3$ and $0 \leq s \leq n-1$. Suppose that $T\subset \R^n$ is a compact set with finite $s$-dimensional upper Minkowski content. Then
\begin{align*}\| \sup_{u\in T} |f*\psi_j*\sigma (\cdot + u)| \|_{p} &\lessapprox 
	2^{-j(\frac{9-4s}{2}-\frac{7-3s}{p} )} \|f\|_{p}
	&(n = 3, \tfrac{3}{2} \leq p \leq 2) \\
	\| \sup_{u\in T} |f*\psi_j*\sigma (\cdot + u)| \|_{p} &\lessapprox  2^{-j(\frac{3n-2-3s}{2}-\frac{2n-1-2s}{p} )} \|f\|_{p}
&(n \geq 4, \tfrac{4}{3} \leq p \leq 2)
\end{align*}
where $A \lessapprox B$ means $A \lesssim_\epsilon 2^{j\epsilon}B$ for any $\epsilon > 0$.
\end{proposition}

For the proof of \Cref{prop:AVG}, we discretize $T$ by using the frequency localization.
\begin{lemma} \label{lem:discrete} Let $T_j$ be the collection of centers of balls of radius $2^{-j}$ covering $T$. Then for $p\geq 1$
\begin{equation}\label{eqn:discretization}
\| \sup_{u\in T}  |f*\psi_j*\sigma (\cdot+u)|\|_p^p \les  \# T_j \|   f*\psi_j*\sigma \|_p^p.
\end{equation}
\end{lemma}
\begin{proof}
Let $\Psi_{j}(x) = 2^{jn}(1+2^j|x|)^{-(n+1)}$ so that $\| \Psi_{j}\|_{L^1} \sim 1$ and $|\tilde{\psi}_j (x)| \les \Psi_{j}(x)$. Note that if $|x-x_0| \les 2^{-j}$, then $\Psi_{j}(x) \les \Psi_{j}(x_0)$. 
Therefore, by using $\psi_j = \psi_j * \tilde{\psi}_j$, we obtain
\begin{equation}\label{eqn:roughlyconst}
\begin{split}
 |f*\psi_j*\sigma (x+u)| &\les |f*\psi_j*\sigma|* \Psi_{j}(x+u_0)\\
 &\les [|f*\psi_j*\sigma|^p* \Psi_{j}(x+u_0)]^{1/p},
\end{split} 
\end{equation} 
whenever $|u-u_0| \les 2^{-j}$. In the last inequality, we have used H\"{o}lder's inequality.

Using \eqref{eqn:roughlyconst}, we get
\begin{align*}
\sup_{u\in T}  |f*\psi_j*\sigma (x+u)|^p \les \sup_{u\in T_j} |f*\psi_j*\sigma|^p*\Psi_{j} (x+u) 
\end{align*}
Therefore, we have
\begin{align*}
\sup_{u\in T}   |f*\psi_j*\sigma(x+u)|^p 
&\les \sum_{u\in T_j}  |f*\psi_j*\sigma|^p*\Psi_{j}(x+u)\\
&\les \sum_{u\in T_j} \int |f*\psi_j*\sigma|^p(x-y+u) \Psi_{j}(y) dy.
\end{align*} 
Integrating the expression over $x$ completes the proof.
\end{proof}

\begin{proof}[Proof of Proposition \ref{prop:AVG}] 
We first recall a standard estimate 
\begin{equation}\label{eq:fixedu}
    \| f*\psi_j*\sigma \|_{L^p} \les 2^{-j(n-1)(1-\frac{1}{p})}\| f\|_{L^p}. 
\end{equation}
for $1\leq p\leq 2$. To see that, one interpolates 
\[ \| f*\psi_j*\sigma \|_{L^1} \les \| f\|_{L^1} \;\; \text{and} \;\; \| f*\psi_j*\sigma \|_{L^2} \les 2^{-j\frac{n-1}{2}}\| f\|_{L^2}, \]
which follow from pointwise estimates (cf. \eqref{eqn:pointwise} and \eqref{eq:stationaryphase})
\begin{equation}\label{eqn:pointwiseN}
|\psi_j*\sigma(x)| \les 2^j(1+2^j| |x|-1|)^{-N} \;\; \text{and} \;\; |\wh{\psi_j*\sigma}(\xi)| \les 2^{-j\frac{n-1}{2}}.
\end{equation}

Next, we cover $T$ with a minimal number of balls of radius $2^{-j}$. Let $T_j$ be the collection of centers of these balls so that $\# T_j = N(T,2^{-j})\les 2^{sj}$. Then by Lemma \ref{lem:discrete} and \eqref{eq:fixedu}, 
\begin{equation}\label{eqn:smallk}
 \| \sup_{u\in T} |f*\psi_j*\sigma(\cdot+u)| \|_p 
\les (\# T_j)^{1/p} \| f*\psi_j*\sigma \|_p  \les 2^{sj/p}
2^{-j(n-1)(1-\frac{1}{p})} \| f\|_p.
\end{equation}
Note that \eqref{eqn:smallk} is one of the claimed bounds. 

We give the other claimed bound by using a different $L^1$-estimate, which is better for $s>1$. The pointwise bound \eqref{eqn:pointwiseN} shows that 
\begin{equation}\label{eqn:pointwise2}
|\psi_j*\sigma(x) |\les 2^j(1+|x|)^{-N}.
\end{equation}
We use this bound to obtain 
\[ |f*\psi_j*\sigma(x+u)| \les 2^j [|f|*(1+|\cdot|)^{-N}] (x+u) \les 2^j [|f|*(1+|\cdot|)^{-N}] (x) \] 
which holds for any $u\in T$ because $T$ is compact. The pointwise estimate implies 
\[ \| \sup_{u\in T} |f*\psi_j*\sigma(\cdot+u)|  \|_1
\les 2^j \| f \|_1. \]
Interpolating this with the $L^2$ bound from \eqref{eqn:smallk} gives the other claimed bound.
\end{proof}

To prove \Cref{prop:AVG-using-NT}, we need, for $\delta \in (0,1/2)$, the $\delta$-regular version of $\MT$, defined by 
\[ \NT f(x) = \sup_{u\in T}  \left|\frac{1}{|S^{\delta}(0)|} \int_{S^\delta(0)} f(x+u+y)dy \right|. \]
Note that $\NT = \NM$ for $T=\S^{n-1}$. We also need the following, which relates $\NT$ and $\MT$.

\begin{lemma}[Relations between $\NT$ and $\MT$] \label{lem:compare} 
Let $0<\delta<1$. Given a function $\phi$, we let $\phi_\delta(x) = \delta^{-n} \phi(x/\delta)$. 

\begin{enumerate}
\item We have
\[ \NT f(x) \les \MT (|f|*\phi_\delta) (x)  \]
for any integrable $\phi \geq 1_{B_2(0)}$ and any integrable $f$. Consequently, 
\[\| \NT \|_{L^p\to L^q} \les \| \MT \|_{L^p\to L^q}.\] 

\item Conversely, if $\| \NT \|_{L^p\to L^q} \leq A(\delta)$ for a non-increasing function $A\gtrsim 1$, then 
\[ \| \MT(f*\phi_\delta) \|_{L^q} \les_\phi A(\delta)\| f \|_{L^p}\]
for any Schwartz function $\phi$ and any $f \in L^p$.
\end{enumerate}
\end{lemma}

Lemma \ref{lem:compare} seems standard; e.g., see \cite[Lemma 5.1]{Schlag_Bo} for a version of the second statement for the circular maximal function. We include a proof for the sake of completeness.

\begin{proof}
Note that 
\begin{equation} \label{eqn:conv}
 \delta^{-1} 1_{S^\delta(0)} \les \delta^{-n} 1_{B_{2\delta}(0)} * \sigma \les \delta^{-1} 1_{S^{2\delta}(0)}.
\end{equation}
Since $\phi \geq 1_{B_2(0)}$, we have $\phi_\delta* \sigma(x)  \ges \delta^{-1} 1_{S^\delta(0)}(x)$ by \eqref{eqn:conv}. Therefore, 
\[ \NT f(x) \les \sup_{u\in T} |f|*\phi_\delta*\sigma (x-u) =   \MT (|f|*\phi_\delta) (x). \]

For the second statement, we note that $\phi_\delta(x)$ decays rapidly away from $B_\delta(0)$. Using a dyadic decomposition, we get
 \[ |\phi_\delta|*\sigma(x) \les 
\sum_{\delta\leq 2^k\delta\leq 1 } 2^{-kN} 
(2^k\delta)^{-1}1_{S^{2^{k+1}\delta}(0)}(x) + \frac{\delta^{N}}{(1+|\cdot|)^N} * \sigma(x).\] by \eqref{eqn:conv}. Therefore, 
\[ \| \MT(f*\phi_\delta) \|_q \leq \| \sup_{u\in T} |f|*|\phi_\delta|*\sigma (\cdot-u) \|_q \les \sum_{\delta\leq 2^k\delta\leq 1 } 2^{-kN} \| \NTk f \|_q + \delta^N \| f\|_p. \]
Since $A(2^k \delta)\leq A(\delta)$ for $k\geq0$ by the assumption, the claim follows from summing over the geometric series.
\end{proof}

\begin{proof}[Proof of \Cref{prop:AVG-using-NT}]

We use the $L^{3/2}$ and $L^{4/3}$ bounds from \Cref{section:NM}. Since $T$ has finite $s$-dimensional upper Minkowski content and $s \leq n-1$, it follows that $T$ has finite $(n-1)$-dimensional upper Minkowski content. By \Cref{remark:NT-same-bounds-as-NM} and \Cref{lem:compare}, we obtain
\begin{align*}\| \sup_{u\in T} |f*\psi_j*\sigma (\cdot + u)| \|_{3/2} &\les j^{1/3} 2^{j/6} \|f\|_{3/2}
	&(n \geq 3) \\
	\| \sup_{u\in T} |f*\psi_j*\sigma (\cdot + u)| \|_{4/3} &\les j^{1/4} 2^{j/4} \|f\|_{4/3}
&(n \geq 4)
\end{align*}
Interpolating these with the $p=2$ case of \eqref{eqn:smallk} gives the desired bounds.
\end{proof}

We note that \Cref{thm:MT} and \Cref{lem:compare} imply bounds for $\| \NT \|_{L^p\to L^p}$ uniform in $\delta$ for a range of $p$ provided that $0\leq s<n-1$. For the case $s=n-1$, we give a Fourier analytic proof of the $L^2$-bound \eqref{eq:NM-L2} by using Proposition \ref{prop:AVG}. 
\begin{theorem}\label{thm:NTfulldim}
Let $0<\delta<1/2$ and $T \subset \R^n$ be a compact set with finite $(n-1)$-dimensional upper Minkowski content. Then 
\[ \| \NT \|_{L^p\to L^p} \les
\begin{cases}
\delta^{-(\frac{2}{p}-1)} & \text{ for } 1\leq p<2 \\
(\log \delta^{-1})^{1/p} & \text{ for } 2\leq p\leq \infty.
\end{cases} \]
\end{theorem}
\begin{proof}
Without loss of generality, we may assume that $f\geq 0$. Let $\phi_\delta(x) = \delta^{-n} \phi(x/\delta)$, where $\phi$ is a function satisfying $\phi \geq 1_{B_2(0)}$ with a compact Fourier support. By \Cref{lem:compare}, it is enough to estimate the $L^p$-norm of $\MT(f*\phi_\delta)$.

We recall the decomposition $\sigma = \varphi_0*\sigma + \sum_{j\geq 1} \tilde{\psi_j}*\psi_j*\sigma$ from Section \ref{sec:Littlewood-Paley}. In the following, we suppress the term $\varphi_0*\sigma$ as it behaves similar to $\tilde{\psi_1}*\psi_1*\sigma$. 
\begin{align*}
 \| \NT f\|_p \les \| \MT(f*\phi_\delta) \|_p &\les \sum_{j\geq 1} \| \sup_{u\in T} |f*\phi_\delta*\tilde{\psi_j}*\psi_j *\sigma(\cdot +u)| \|_p \\
 &\les \sum_{j\geq 1} 2^{j(\frac{2}{p}-1)} \| f*\phi_\delta*\tilde{\psi_j} \|_p,
\end{align*}
where the sum in $j$ is a finite sum for $2^j \les \delta^{-1}$ due to the frequency localization of $\phi_\delta$. We have used Proposition \ref{prop:AVG} in the last inequality. Since $\| f*\phi_\delta*\tilde{\psi_j} \|_p \les \| f \|_p$, this yields the claim for $1\leq p<2$. 

For $p=2$, we use Cauchy-Schwarz and Plancherel to get
\begin{align*}
 \| \NT f\|_2 \les \sum_{j} \| f*\phi_\delta*\tilde{\psi_j} \|_2 &\les (\log \delta^{-1})^{1/2} \Big( \sum_{j}  \| f*\phi_\delta*\tilde{\psi_j} \|_2^2 \Big)^{1/2} \\ &\les (\log \delta^{-1})^{1/2}  \| f \|_2. 
 \end{align*}
The case $p>2$ follows from interpolation with the trivial $L^\infty$ bound.
\end{proof}

\begin{remark}
For $n=2$ and the range $1 < p < 2$, \Cref{thm:NTfulldim} yields slightly better bounds than the proof in \Cref{section:NM}, since here, there is no logarithmic factor. 
\end{remark}

\subsection{Spherical averages relative to fractal measures}
In this section, we relate \eqref{eqn:STreduction} with spherical averages relative to fractal measures. We recall that $\mathcal{C}(\alpha)$ is defined in \eqref{eqn:alpha_dimensional_measure}.

\begin{lemma}\label{cor:linearize}
Let $T\subset \R^n$ be a compact set with finite $s$-dimensional upper Minkowski content for some $0\leq s < n-1$. Then \eqref{eqn:STreduction} holds if
\begin{equation}\label{eqn:smoothingfractal}
   \| \avg(f*\psi_j) \|_{L^p(\nu)} \leq A \| f\|_{L^p(\R^n)} 
\end{equation}
holds for any $\nu\in \mathcal{C}(\alpha)$ with $\alpha=n-s$ for some $A>0$ independent of $f$ and $\nu$.
\end{lemma}
\begin{proof}
First, we note that it suffices to prove a local estimate of the maximal function on a unit ball. This can be seen from the decomposition $f = \sum_{Q} f_Q$, where $f_Q$ is supported on a lattice unit cube $Q$ and then observing that the function $|f_Q*\psi_j*\sigma_t (\cdot + tu)|$ decays rapidly away from $O(2^{c\epsilon j})$ neighborhood of $Q$ for some $c>0$. See, for instance, the proof of Lemma 2.4 in \cite{HKL} for details.  

We use the Kolmogorov-Seliverstov-Plessner linearization, \emph{cf}. \cite[Chapter XIII]{Zygmund}. For given measurable functions $t : \R^n \to [1,2]$ and $u: \R^n \to T$, define
\begin{equation*}
   T_j f(x) :=  f*\psi_j*\sigma_{t(x)} (x+ t(x)u(x)) = \avg (f*\psi_j)(x+t(x)u(x),t(x)).
\end{equation*}
Then it suffices to prove that there exists a measure $\nu \in \mathcal{C}(\alpha)$ for $\alpha = n-s$, such that 
\[ \| T_j f \|_{L^p(B^n_1)} \les \| \avg(f*\psi_j) \|_{L^p(\nu)}, \]
where the implicit constant is independent of $t, u$ and $f$.

Let $l$ be a positive linear functional on $C_c(\R^{n}\times \R)$ defined by
\[ l(F) = \int F(x+t(x)u(x), t(x)) dx.\]
Then by the Riesz representation theorem, there exists a unique Radon measure $\tilde{\nu}$ on $\R^n\times \R$ such that 
\[ l(F) = \int F(x,t) d\tilde{\nu}(x,t).\]

Let $\chi\geq 0$ be a smooth function such that $1_{B^n_C} \leq \chi \leq 1_{B^n_{2C}}$ for a sufficiently large absolute constant $C=C_T>0$. With $F(x,t)= |\avg (f*\psi_j) (x,t)|^p \chi(x)\eta(t)$ for a smooth cutoff function $1_{[1,2]} \leq  \eta \leq 1_{[1/2,4]}$, we have
\[ \int_{B_1^n} |T_jf|^p  \leq l(F) \sim \int |\avg (f*\psi_j) (x,t)|^p d\nu(x,t), \]
where we set $d\nu(x,t) = c \chi(x)\eta(t) d\tilde{\nu}(x,t)$ for some small absolute constant $c=c_{n,s,T}>0$.

It remains to verify that $\nu \in \mathcal{C}(\alpha)$ for $\alpha=n-s$ provided that $c>0$ is sufficiently small (cf. \Cref{lem:count}). To see this, fix a ball $B^{n+1}_r$ centered at $(x_0,t_0) \in \R^n\times  \R$ of radius $0<r\leq 1$. We may further assume that $0<r\leq 1/2$ since any ball of radius comparable to 1 is a union of $O_n(1)$ balls of radius $1/2$. Observe that 
\[\nu(B^{n+1}_r)  \leq c | \{ x \in \R^n : |x+t(x)u(x)- x_0|^2+ |t(x)-t_0|^2 \leq r^2  \}|. \]
We may assume that $t_0 \sim 1$ since otherwise the set under consideration is empty. Note that if $|x+t(x)u(x)- x_0|^2+ |t(x)-t_0|^2 \leq r^2 $, then by the triangle inequality and the assumption that $T$ is compact,
\[ |x+t_0 u(x) - x_0| \leq |x+ t(x) u(x) - x_0| + |(t(x)-t_0) u(x)| \les r, \]
so $x \in x_0 - \nbd{t_0 T}{O(r)}$. Therefore, 
 \[ \nu(B^{n+1}_r) \leq c| x_0 - \nbd{t_0 T}{O(r)}| = c|\nbd{t_0 T}{O(r)}|  \les cr^{n} N(t_0T,O(r))\les cr^{n-s}. \] 
This finishes the proof.
\end{proof}

We discuss the connection between spherical averages and the half-wave propagator $e^{it\sqrt{-\Delta}}$ associated with the Fourier multiplier $e^{it|\xi|}$. By the method of stationary phase (see e.g. \cite{Sogge}), one may write
\begin{equation}\label{eq:stationaryphase}
   \ft{\sigma}(\xi) =  e^{i|\xi|} a_{+}(\xi) + e^{-i|\xi|} a_{-}(\xi), 
\end{equation}
for $a_\pm$ satisfying $|\partial^m a_\pm(\xi) | \les (1+|\xi|)^{-\frac{n-1}{2} -|m| }$ for each multi-index $m$. Therefore, for $t\in [1,2]$, we may write
\begin{equation}\label{eq:avgtowave}
\begin{split}
\avg(f*\psi_j)(x,t) = 2^{-j\frac{n-1}{2}} \sum_{\pm} \int e^{i x\cdot \xi \pm i t|\xi| } a_{j,\pm}(t,\xi) \ft{f*\psi_j}(\xi) d\xi
\end{split}
\end{equation}
where $a_{j,\pm}(t,\xi) := 2^{j(n-1)/2}a_\pm(t\xi)\ft{\tilde{\psi_j}}(\xi) \eta(t)\in C^\infty(\R^n\times [1/4,4] )$ for a smooth bump function $\eta$ supported on $[1/2,4]$. We note that $a_{j,\pm}(t,\xi)$ is supported on $\{(\xi,t): |\xi|\sim 2^j, t\sim 1 \}$  and satisfies 
\begin{equation}\label{eq:symbolbound}
 |\partial_{t}^l \partial_{\xi}^m a_{j,\pm}(t,\xi)| \les_{m, l} 2^{-|m|j}
\end{equation}
for every $l\geq 0$ and multi-index $m$. Since the contribution from $+$ and $-$ terms in \eqref{eq:avgtowave} can be handled similarly, from now on we shall ignore the contribution from the $-$ term.

We state a version of Sogge's local smoothing conjecture for the wave equation to be used later. Let $I\subset (0,4)$ be a compact interval. The conjecture asserts that 
\begin{equation}\label{eqn:smoothing0}
\| e^{it\sqrt{-\Delta}} (f*\psi_j) \|_{L^p(\R^{n} \times I)} \less 2^{j\alpha(p)} \| f \|_{L^p(\R^n)},
\end{equation}
where $\alpha(p)= \frac{n-1}{2} - \frac{n}{p}$ for $p\geq  2n/(n-1)$ and $\alpha(p)=0$ for $2\leq p \leq 2n/(n-1)$. We will later use, in \Cref{prop:duality}, the fact that the local smoothing estimate \eqref{eqn:smoothing0} holds for \emph{some} $p\geq 2n/(n-1)$; we may take, for instance, $p= 2(n+1)/(n-1)$ for every $n\geq 2$ by the Bourgain--Demeter decoupling inequality for the cone \cite{BD}. We refer the reader to \cite{Beltran-Hickman-Sogge-exposition} for a nice survey of the theory of Fourier integral operators and local smoothing estimates.

\Cref{prop:reduce} and \Cref{cor:linearize} reduce $L^p$-estimates for $\ST$ to estimates for spherical averages relative to fractal measures \eqref{eqn:smoothingfractal}. In view of \eqref{eq:avgtowave}, such estimates can be deduced from local smoothing estimates \eqref{eqn:smoothing0} with the Lebesgue measure on $\R^n\times I$ replaced by measures $\nu \in \mathcal{C}(\alpha)$. We recall some special cases of known results in this direction from \cite{CHL,HKL}. 

\begin{theorem}[{\cite[Theorem 3.1]{CHL}}, see also {\cite[Theorem 3.3]{HKL}}] \label{thm:CHL} 
Let $n\geq 3$, $1<\alpha\leq n$ and $\nu \in \mathcal{C}(\alpha)$. Then 
\[ 
    2^{-j(n-1)/2} \| e^{it\sqrt{-\Delta}} (f*\psi_j) \|_{L^2(\nu)} \less 2^{-j\min( \frac{\alpha-1}{2}, \frac{n+2\alpha-5}{8})} \| f\|_{L^2(\R^n)}.
\]
\end{theorem}
We note that the bound is sharp when $n=3$ or $1<\alpha \leq \frac{n-1}{2}$; see \cite{CHL}.

\begin{theorem}[{\cite[Theorem 1.5]{HKL}}] \label{thm:HKL} Let $n=2$, $1<\alpha\leq 2$ and $\nu \in \mathcal{C}(\alpha)$. Then there exists $\epsilon = \epsilon(\alpha,p) > 0$ such that
\[
   2^{-j/2}\| e^{it\sqrt{-\Delta}} (f*\psi_j) \|_{L^p(\nu)} \les  2^{-j\epsilon } \| f\|_{L^p(\R^2)} 
\]
for $p>\max(4-\alpha, (6-2\alpha)/\alpha)$.
\end{theorem}
For the proof of \Cref{thm:ST}, we additionally need the following estimates which is better than \Cref{thm:HKL} for $1<\alpha<6/5$ for $n=2$ and \Cref{thm:CHL} for $\alpha > (n+3)/2$ for $n\geq 3$. 

\begin{proposition}\label{prop:alpha1} Let $n=2$, $1<\alpha\leq 3$ and $\nu \in \mathcal{C}(\alpha)$. Then there exists $\epsilon = \epsilon(\alpha,p) > 0$ such that
\begin{equation}\label{eqn:fractalavg2d}
   \| \avg (f*\psi_j) \|_{L^p(\nu)} \les  2^{-j\epsilon } \| f\|_{L^p(\R^2)} 
\end{equation}
for any $p>3$. Moreover, for $n\geq 2$, 
\begin{equation}\label{eqn:fractalavg}
     \| \avg (f*\psi_j) \|_{L^2(\nu)} \les 2^{-j \frac{\alpha-2}{2}} \| f\|_{L^2(\R^n)}.
\end{equation}
\end{proposition}
\Cref{prop:alpha1} will be proved in \Cref{sec:duality} using the geometric input from \Cref{prop:M2}.

\subsection{A geometric approach to fractal local smoothing estimates} \label{sec:duality}
In this section, we prove \Cref{prop:alpha1}, \Cref{thm:ST} and \Cref{thm:avg_frac}. The proof relies on a geometric characterization of fractal local smoothing estimates. We first recall the following basic properties of measures in $\mathcal{C}(\alpha)$ defined in \eqref{eqn:alpha_dimensional_measure}.

\begin{lemma}[cf.~Lemmas 2.7 and 3.1 in \cite{HKL}]\label{lem:Calpha}
    Let $\nu \in \mathcal{C}(\alpha)$, $N>n+1$ and $1\leq q <\infty$. If $F$ is a function whose Fourier transform is supported on $B^{n+1}(0, O(\delta^{-1}))$ for some $\delta\in (0,1/2)$, then  
    \begin{equation} \label{eq:measuretoweight}
        \| F \|_{L^q(\nu)} \les \| F \|_{L^q(\nu*\Psi_{\delta,N})},
    \end{equation}
  where  $\Psi_{\delta,N} (z) = \delta^{-(n+1)}(1+\delta^{-1}|z|)^{-N}$ for $z\in \R^{n+1}$. Moreover, 
    \begin{enumerate}[label=(\roman*)]
        \item $\| \nu*\Psi_{\delta,N} \|_{L^\infty} \les \delta^{\alpha-(n+1)}$ \label{eq:measuresup}
        \item $\| \nu*\Psi_{\delta,N} \|_{L^\infty(\R^n\times [2^{-1},5\cdot 2^{-1}]^c)} \les \delta^{N-(n+1)}$  \label{eq:measuresupexceptional}
        \item $\int_{B_r} \nu*\Psi_{\delta,N} \les r^\alpha$ for any ball $B_r \subset \R^{n+1}$ of radius $r\in [\delta, 1]$. \label{eq:measureball}
    \end{enumerate}
In the above estimates, implicit constants may depend on $n,N$ and $\alpha$. 
\end{lemma}

 From now on, we fix $\delta = 2^{-j} \ll 1$ and write $\psi_\delta := \psi_j$. Note that the Fourier transform of the function $(x,t) \mapsto e^{it\sqrt{-\Delta}} (f*\psi_\delta)(x)$ is supported on a truncated cone contained in a ball of radius $\sim \delta^{-1}$. Thus, we obtain the following corollary.
\begin{corollary}\label{eqn:removemesure} Let $\nu \in \mathcal{C}(\alpha)$ and $1\leq q  < \infty$. Then 
\[
\| e^{it\sqrt{-\Delta}} (f*\psi_\delta) \|_{L^{q}(\nu)} \les \delta^{(\alpha-(n+1))/q} \| e^{it\sqrt{-\Delta}} (f*\psi_\delta) \|_{L^{q}(\R^{n} \times I)}. \]
\end{corollary}

   \Cref{lem:Calpha} essentially follows from proofs of Lemmas 2.7 and 3.1 in \cite{HKL}. However, our definition of the class of measures $\mathcal{C}(\alpha)$ is slightly different from that paper in that we only require $0<r\leq 1$ in the condition \eqref{eqn:alpha_dimensional_measure} which necessitates a few minor modifications. For completeness, we sketch the proof. 
    
\begin{proof}[Proof of \Cref{lem:Calpha}]   
    Inequality \eqref{eq:measuretoweight} follows from the fact that, by the Fourier localization of $F$, we have $|F|^q=|F*\psi_\delta|^q \les |F|^q*|\psi_\delta|$ for a smooth function $\psi_\delta$ satisfying $|\psi_\delta(z)|\les_N \Psi_{\delta, N} (z)$.

    For the bounds on $\nu*\Psi_{\delta,N}$, we use the dyadic decomposition
    \begin{equation}\label{eqn:measure_dyadic}
       \nu*\Psi_{\delta,N}(z) \les \delta^{-(n+1)} \sum_{l=0}^\infty 2^{-lN} \nu(B(z,2^l \delta)).
    \end{equation}
    In the sum over $l$, we consider the following parts separately: a) $\delta \leq 2^l \delta \leq 1$ and b) $2^l \delta > 1$. For part a), we use $\nu(B(z,r)) \leq r^\alpha$ for $r\leq 1$, while for part b), we use $\nu(B(z,r)) \leq Cr^{n+1}$ for $r\geq 1$, which follows from covering $B(z,r)$ by balls of radius 1. Combining these estimates, we obtain \ref{eq:measuresup}. For \ref{eq:measuresupexceptional}, we need to use the additional fact that if $z\in \R^n \times [2^{-1},5\cdot 2^{-1}]^c$, then $\nu(B(z,2^l \delta))=0$ whenever $2^l \delta \leq 2^{-1}$ since $\nu$ is supported on $\R^n\times [1,2]$. 

    For part \ref{eq:measureball}, we consider the following two cases separately: a) $\delta \leq 2^l\delta \leq 1$ and b) $2^l \delta > 1$. Note that 
    \[ \int_{B_r} \nu(B(z,2^l\delta)) dz = \int |B(z',2^l\delta) \cap B_r|  d\nu(z'), \]
    where the Lebesgue measure $|B(z',2^l\delta) \cap B_r|$ is 0 unless $z'$ belongs to the ball of radius $r+2^l\delta$ centered at the center of $B_r$. For part a), we use $\int |B(z',2^l\delta) \cap B_r|  d\nu(z') \les \min(2^l\delta,r)^{n+1} \max(2^l\delta,r)^\alpha$. For part b), we use $\int_{B_r} \nu(B(z,2^l\delta)) dz \les (2^l\delta)^{n+1} r^{n+1}\leq (2^l\delta)^{n+1} r^{\alpha}$. Summing over $l$ in each case and applying \eqref{eqn:measure_dyadic} yields the claimed bound.
\end{proof}

For each $z=(x,t) \in \R^n\times [1,2]$, let $z^\delta$ and $S^\delta(x,t)$ denote the $\delta$-neighborhood of the sphere $S(x,t)=\{y\in \R^n: |x-y| = t\}$. For $0< \alpha \leq n+1$ and $A>0$, let $\mathcal{X}(\alpha, \delta, A)$ be the class of $\delta$-separated set of points $X$ in $\R^{n+1}$ such that 
\begin{equation}\label{eq:density}
 \# (X \cap B_r) \leq  Ar^\alpha,
\end{equation}
for any ball $B_r\subset \R^{n+1}$ of radius $r \in (\delta, 1]$ and $X^\delta \subset \R^{n}\times [1,2]$, where $X^\delta$ denotes the $\delta$-neighborhood of $X$. Here $A$ may depend on $\delta$, and we may assume that $0<A\les \delta^{-(n+1)}$.

\begin{lemma}\label{prop:duality}  Let $n\geq 2$,  $0<\delta \ll 1$, $0<\alpha \leq n+1$, and $1\leq p\leq q \leq \infty$. Consider the following statements.
\begin{align}
\| \sum_{z \in X} 1_{z^{\delta}} \|_{L^{p'}} &\less \delta A^{1/q} ( \# X)^{1/q'}, &\text{ for any } X\in\mathcal{X}(\alpha,\delta, A). \label{eq:dual1}\\
\| \sum_{z \in X} a_z 1_{z^{\delta}} \|_{L^{p'}} &\less \delta A^{1/q}  \left( \sum_{z\in X} |a_z|^{q'} \right)^{1/q'}, & \text{ for any } X\in\mathcal{X}(\alpha,\delta, A). \label{eq:dual} \\
 \delta^{(n-1)/2} \| e^{it\sqrt{-\Delta}} (f*\psi_\delta) \|_{L^q(X^\delta)} &\less (A\delta^{n+1})^{1/q}  \| f\|_{L^p(\R^n)},  &\text{ for any } X\in\mathcal{X}(\alpha,\delta, A). \label{eq:weight1} \\
  \delta^{(n-1)/2} \| e^{it\sqrt{-\Delta}} (f*\psi_\delta) \|_{L^q(\nu)} &\less  \| f\|_{L^p(\R^n)},  &\text{ for any } \nu \in\mathcal{C}(\alpha). \label{eq:measure} \\
   \| \avg (f*\psi_\delta) \|_{L^{q}(\nu)} &\less  \| f\|_{L^{p}(\R^n)},  &\text{ for any } \nu \in\mathcal{C}(\alpha). \label{eq:average_localized} 
\end{align}
Then $\eqref{eq:dual1} \implies \eqref{eq:dual} \implies \eqref{eq:weight1} \implies \eqref{eq:measure} \implies \eqref{eq:average_localized}$. 

Let $\frac{2n}{n-1} \leq p_n < \infty $ be an exponent for which the $L^{p_n}(\R^n) \to L^{p_n}(\R^n\times I)$ local smoothing estimate \eqref{eqn:smoothing0} holds. When $\alpha>1$, all the above statements are equivalent to the following:  for any $(1/\tilde{p},1/\tilde{q})$ in the open line segment connecting $(1/p,1/q)$ and $(1/p_n,1/p_n)$, there exists $\epsilon>0$ such that
\begin{align}
  \| \avg (f*\psi_\delta) \|_{L^{\tilde{q}}(\nu)} &\les \delta^{\epsilon} \| f\|_{L^{\tilde{p}}(\R^n)},  &\text{ for any } \nu \in\mathcal{C}(\alpha). \label{eq:average}
\end{align}
\end{lemma}
\begin{remark} 
\label{remark:duality-lemma-constants}
   For any $M=M_{\delta,n,\alpha,p,q}>0$, the chain of implications from \eqref{eq:dual1} to \eqref{eq:average_localized} holds under the replacement of $\less$ by $\less M$ in each inequality. Moreover, the implication $\eqref{eq:measure} \implies \eqref{eq:average_localized}$ holds with $\less$ replaced by $\les M$ in both inequalities.
\end{remark}
\begin{remark}
 We may take any finite $p_n \geq 2(n+1)/(n-1)$ for the exponent $p_n$ in \Cref{prop:duality} for every $n\geq 2$; see the discussion following \eqref{eqn:smoothing0}. 
\end{remark}
\begin{remark}
    Without loss of generality, we may assume that each $X \in \mathcal{X}(\alpha, \delta, A)$ is contained in a unit ball in this section. This is because \eqref{eq:dual1} is equivalent to its local version concerning $X\in \mathcal{X}(\alpha, \delta, A)$ contained in a unit ball. \end{remark}

Before proving \Cref{prop:duality}, we give proofs of \Cref{prop:alpha1}, \Cref{thm:ST} and \Cref{thm:avg_frac}.
\begin{proof}[Proof of \Cref{prop:alpha1}]
Let $n=2$. We have \eqref{eq:dual1} for $p=q=3$ by \Cref{prop:M2} for any $1\leq \alpha\leq 3$. Fix $p>3$ and $p_n>p$ such that the local smoothing estimate \eqref{eqn:smoothing0} holds at the exponent $p_n$. Then \eqref{eqn:fractalavg2d} follows from \Cref{prop:duality} for $1<\alpha\leq 3$.

Next, let $n\geq 2$. By the implication $\eqref{eq:measure} \implies \eqref{eq:average_localized}$, \eqref{eqn:fractalavg} follows from
\begin{equation}\label{eqn:prop:duality}
\| e^{it\sqrt{-\Delta}} (f*\psi_j) \|_{L^2(\nu)} \les 2^{j(n+1-\alpha)/2} \| f \|_{L^2(\R^2)}.
\end{equation}
Note that \eqref{eqn:prop:duality} is a consequence of \Cref{eqn:removemesure} (and \Cref{remark:duality-lemma-constants}) and the Plancherel theorem.

\end{proof}

\begin{proof}[Proof of \Cref{thm:ST}]
Let $n=2$. In view of \Cref{prop:duality}, we know that \eqref{eqn:fractalavg2d} holds for $p>\min\left(3,\max\left(4-\alpha,(6-2\alpha)/\alpha \right) \right)$  by combining \Cref{thm:HKL} and \Cref{prop:alpha1}. By \Cref{prop:reduce} and \Cref{cor:linearize}, $\ST$ is bounded on $L^p(\R^2)$ for the same $p$-range with $\alpha=2-s$, which completes the proof for $n=2$.

Let $n\geq 3$. By \Cref{prop:reduce}, \Cref{cor:linearize} and \Cref{prop:duality}, estimates from \Cref{thm:CHL} and \Cref{prop:alpha1} imply 
\begin{align*}
    &\| \sup_{l\in \Z} \sup_{u\in T} \sup_{1\leq t\leq 2} |\avg(f*\psi_{j-l})(\cdot + 2^ltu,2^lt)| \|_{L^2(\R^n)} \\
    &\less 2^{-j\max( \frac{\alpha-2}{2}, \min( \frac{\alpha-1}{2}, \frac{n+2\alpha-5}{8}) )} \| f \|_{L^2(\R^n)}.
\end{align*}
for $\alpha=n-s$. 
By interpolation with the weak-type estimate in \Cref{lem:L1}, there exists $\epsilon=\epsilon(p)>0$ 
such that
\[ 
    \| \sup_{l\in \Z} \sup_{u\in T} \sup_{1\leq t\leq 2} |\avg(f*\psi_{j-l})(\cdot + 2^ltu,2^lt)| \|_{L^p(\R^n)} \les 2^{-j\epsilon} \| f\|_{L^p(\R^n)}
\]
for $p> 1+ ( \max(n-s-1 , \min(n-s, (3n-2s-1)/4)) )^{-1} $. Therefore, by \Cref{prop:reduce}, $\ST$ is bounded on $L^p$ for the same $p$-range.
\end{proof}

\begin{proof}[Proof of \Cref{thm:avg_frac}]
    We recall that the solution $u$ can be written as 
    \[ u(x,t) = \cos(t\sqrt{-\Delta}) f(x) =  \frac{1}{2} e^{it\sqrt{-\Delta}} f(x) + \frac{1}{2} e^{-it\sqrt{-\Delta}} f (x). \]
Therefore, \eqref{eqn:fraclo} is a consequence of \Cref{prop:M2} and \Cref{prop:duality}. The inequality \eqref{eqn:avg_frac} follows from \Cref{prop:alpha1} by summing up frequency localized estimates \eqref{eqn:fractalavg2d}.
\end{proof}

The proof of \Cref{prop:duality} is fairly standard. We will use a pointwise bound for the kernel associated with the wave propagator, which can be obtained by using \eqref{eq:stationaryphase};
\[ e^{it\sqrt{-\Delta}} (f*\psi_\delta)(x) = \int f(x-y)\delta^{-n}K(\delta^{-1}y,\delta^{-1}t) dy, \] where $K$ satisfies
\begin{equation}\label{eq:kernel}
|K(y,t)| \les_N (1+|y|)^{-(n-1)/2} (1+||y|-t|)^{-N}.
\end{equation}
We note that the standard bound \eqref{eq:kernel} follows from the decay and the oscillation of the Fourier transform of the spherical measure; see e.g., \cite{Sogge}.

\begin{proof}[Proof of \Cref{prop:duality}]

$\eqref{eq:dual1} \implies \eqref{eq:dual}$:
An interpolation argument can be used to upgrade the restricted strong type estimate from \eqref{eq:dual1} to a strong type estimate 
\eqref{eq:dual}. To see that, we first note the following trivial $l^{1} \to L^1$ and $l^{r'} \to L^\infty$ estimates;
\begin{align}
\| \sum_{z \in X} a_z 1_{z^{\delta}} \|_{L^{1}} &\les \delta  \sum_{z \in X} |a_z| \label{eq:trivialL1} \\
\| \sum_{z \in X} a_z 1_{z^{\delta}} \|_{L^{\infty}} &\les \sup_{B_1}  \sum_{z \in X\cap B_1} |a_z| \les \| \{a_z\} \|_{l^{r'}(X)}  \sup_{B_1} \#(X\cap B_1)^{1/r} \nonumber \\
&\labelrel\lesssim{eq:density}
A^{1/r} \| \{a_z\} \|_{l^{r'}(X)}. \label{eq:trivialLinfty}
\end{align}
We choose $r=q/p \geq 1$ so that the line segment between $(1,1)$ and $(1/r',1/\infty)$ contains $(1/q',1/p')$. 
We first do a real interpolation between \eqref{eq:dual1} and \eqref{eq:trivialLinfty} to get a strong type $l^{q_1'} \to L^{p_1'}$ estimate, and then interpolate that estimate with \eqref{eq:trivialL1}. As we may take $(1/q_1',1/p_1')$ arbitrarily close to $(1/q',1/p')$, these interpolations imply \eqref{eq:dual}, since we are allowed to lose $\delta^{-\epsilon}$.

$\eqref{eq:dual} \implies \eqref{eq:weight1}$:
For any given $0<\epsilon \ll 1$, \eqref{eq:kernel} (with $N=100n/\epsilon$) implies that
\[ \delta^{(n-1)/2} |e^{it\sqrt{-\Delta}} (f*\psi_\delta) (x)| \les_\epsilon \delta^{-1} \int 1_{S^{\delta^{1-\epsilon}}(x,t)} |f| + \delta^{10n} f*(1+|\cdot|)^{-{10n}} (x). \]
The term involving $\delta^{10n}$ is harmless and will be ignored in the following. We have

\begin{align*}
\delta^{(n-1)q/2} \|e^{it\sqrt{-\Delta}} (f*\psi_\delta)\|_{L^q(X^\delta)}^q &\les_{\epsilon} \delta^{-q}  \sum_{z \in X} \delta^{n+1} \left( \int 1_{z^{O(\delta^{1-\epsilon})}} |f| \right)^q \\
&= \delta^{-q} \delta^{n+1}  \left(\int |f| \sum_{z\in X} a_z 1_{z^{O(\delta^{1-\epsilon})}} \right)^q \\
&\leq \delta^{-q} \delta^{n+1}  \| f\|_{L^p}^q \| \sum_{z\in X} a_z 1_{z^{O(\delta^{1-\epsilon})}} \|_{L^{p'}}^q 
\end{align*}
for some $a_z\geq0$ satisfying $\sum_{z\in X}  a_z^{q'} = 1$. Applying \eqref{eq:dual}, we obtain \eqref{eq:weight1}.

$\eqref{eq:weight1} \implies \eqref{eq:measure}$: By \Cref{lem:Calpha}, it suffices to prove \eqref{eq:measure} with $\nu$ replaced by $\nu*\Psi_{\delta,N}$ for some $N>n+1$. Note that $\nu*\Psi_{\delta,N}$ is essentially constant on $\delta$-balls in the sense that $\nu*\Psi_{\delta,N}(z) \sim \nu*\Psi_{\delta,N}(z')$ for $|z-z'|\leq \delta$. Therefore, we have
\begin{equation}\label{eq:decom}
\nu*\Psi_{\delta,N} \sim \sum_{l\in \Z} 2^l 1_{X_l^\delta}, 
\end{equation}
for some $\delta$-separated set of points $X_l \subset \R^{n+1}$. We may assume that $X_l$ is empty when $2^l \les \delta^C$ for some sufficiently large $C$, say $C=10n$, since the sum over $2^l \les \delta^C$ can be handled by using a trivial inequality. Moreover, we note that $X_l$ is empty when $2^l \gtrsim \delta^{-C}$ by \ref{eq:measuresup} of \Cref{lem:Calpha}. Thus, it suffices to deal with the sum over $O(\log \delta^{-1})$ many $l$ satisfying $\delta^C  \les 2^l \les \delta^{-C}$. We may assume, by taking $N\geq C+n+1$, that $X_l^\delta\subset \R^n \times [2^{-1}, 5\cdot 2^{-1}]$ by \ref{eq:measuresupexceptional} of \Cref{lem:Calpha}. Thus, there exists $l$ for which 
\begin{equation}\label{eq:pigeon}
  \| e^{it\sqrt{-\Delta}} (f*\psi_\delta) \|_{L^q(\nu*\Psi_{\delta,N})}  \less 2^{l/q} \| e^{it\sqrt{-\Delta}} (f*\psi_\delta) \|_{L^q(X_l^\delta)}. 
\end{equation}
Note that for $\delta<r\leq 1$, 
\[  \delta^{n+1} \#(X_l\cap B_r) \les |X_l^\delta \cap B_{r+\delta}| \les 2^{-l} \int_{X_l^\delta \cap B_{r+\delta}}  \nu*\Psi_{\delta,N} \les 2^{-l} r^\alpha, \]
by \ref{eq:measureball} of \Cref{lem:Calpha}.
Applying \eqref{eq:weight1} with $A\delta^{n+1}  \sim 2^{-l}$ in \eqref{eq:pigeon} finishes the proof of the implication. To be precise, we have to apply a slightly more general version of \eqref{eq:weight1} which holds for $\delta$-separated set of points $X$ satisfying \eqref{eq:density} such that $X^\delta \subset \R^n \times [2^{-1}, 5\cdot 2^{-1}]$. Note that \eqref{eq:dual1} can be extended to such a class of sets by scaling, which implies similar extensions for \eqref{eq:dual} and \eqref{eq:weight1} required. 

$\eqref{eq:measure} \implies \eqref{eq:average_localized}$:
Recall that we ignore the $-$ term in the sum in \eqref{eq:avgtowave}. In view of \eqref{eq:avgtowave}, this implication would be straightforward if we are allowed to replace $a_{j,+}(t,\xi)$ by $\eta(t)\ft{\tilde{\psi_j}}(\xi)$ as the term $\ft{\tilde{\psi_j}}(\xi)$ can be then absorbed into $f$ as $f*\tilde{\psi_j}$. To handle $a_{j,+}(t,\xi)$, we may employ an argument which is used to prove the $L^2$-boundedness of pseudo-differential operators; see e.g. \cite[Chapter VI \se 2]{Stein}. We give the argument in our setting for the convenience of the reader. 

For $\tau\in \R$, define $\eta_{j,\tau}(x) = \mathcal{F}^{-1} [a_{j,+}](\tau,x)$, so that 
\[ a_{j,+}(t,\xi) = \int \ft{\eta_{j,\tau}}(\xi) e^{-it\tau} d\tau. \]
Note that the integral in \eqref{eq:avgtowave} can now be written as 
\[ \int e^{it\sqrt{-\Delta}} (\eta_{j,\tau}*f*\psi_j) e^{-it\tau} d\tau.  \]
Therefore, by \eqref{eq:avgtowave},  \eqref{eq:measure}, and Minkowski's inequality,
\begin{align*}
   \| \avg (f*\psi_\delta) \|_{L^{q}(\nu)} &\les 2^{-j\frac{n-1}{2}} \int \| e^{it\sqrt{-\Delta}} (\eta_{j,\tau}*f*\psi_\delta)\|_{L^{q}(\nu)} d\tau \\
   &\less \int \| \eta_{j,\tau}*f \|_{L^p(\R^n)} d\tau \leq  \| f\|_{L^p(\R^n)} \int \| \eta_{j,\tau}\|_{L^1(\R^n)} d\tau.
\end{align*}
To complete the proof, it remains to observe that $\| \eta_{j,\tau} \|_{L^1(\R^n)}\les (1+|\tau|)^{-2}$, 
which is a consequence of the pointwise estimate obtained by integration by parts using \eqref{eq:symbolbound}: 
\[ |\eta_{j,\tau}(x) | \les 2^{jn}(1+2^j|x|)^{-(n+1)}(1+|\tau|)^{-2}.\] 
This finishes the proof of the chain of implications from \eqref{eq:dual1} to \eqref{eq:average_localized}.

Next, we assume that $\alpha>1$. Suppose that the local smoothing estimate \eqref{eqn:smoothing0} is available at the exponent $2n/(n-1)\leq p_n < \infty$, i.e.,
\begin{equation}\label{eqn:smoothing0n}
\delta^{(n-1)/2} \| e^{it\sqrt{-\Delta}} (f*\psi_\delta) \|_{L^{p_n}(\R^{n} \times [1,2])} \less \delta^{n/p_n}  \| f \|_{L^{p_n}(\R^n)}.
\end{equation}
By \Cref{eqn:removemesure} and \eqref{eqn:smoothing0n}, we obtain
\[ \delta^{(n-1)/2} \| e^{it\sqrt{-\Delta}} (f*\psi_\delta) \|_{L^{p_n}(\nu)} \less \delta^{(\alpha-1)/p_n}  \| f \|_{L^{p_n}(\R^n)}. \]
By the implication $\eqref{eq:measure} \implies \eqref{eq:average_localized}$, we have
\begin{equation}\label{eqn:avgfromsmoothing}
   \| \avg(f*\psi_\delta) \|_{L^{p_n}(\nu)} \less  \delta^{(\alpha-1)/p_n}  \| f \|_{L^{p_n}(\R^n)} .
\end{equation}
To show equivalence of these statements, it remains  to verify that $\eqref{eq:average_localized} \implies \eqref{eq:average} \implies \eqref{eq:dual1}$.

$\eqref{eq:average_localized} \implies \eqref{eq:average}$: 
Since $\alpha>1$, an interpolation of \eqref{eq:average_localized} and \eqref{eqn:avgfromsmoothing} gives \eqref{eq:average}.

$\eqref{eq:average} \implies \eqref{eq:dual1}$: 
By summing over dyadic $\delta=2^{-j}$ over $j\geq 0$, \eqref{eq:average} implies 
 \begin{align}
  \| \avg f \|_{L^{\tilde{q}}(\nu)} &\les \| f\|_{L^{\tilde{p}}(\R^n)},  &\text{ for any } \nu \in\mathcal{C}(\alpha).  \label{eq:fullaverage}
 \end{align}
It is enough to prove that \eqref{eq:fullaverage} implies \eqref{eq:dual} since \eqref{eq:dual} implies \eqref{eq:dual1}. By duality, it suffices to show that
\begin{align}\label{eq:dualstep}
\left(\sum_{z\in X} \left(\frac{1}{\delta} \int_{z^\delta} |f|\right)^q  \right)^{1/q} \less A^{1/q}   \| f \|_{L^p} & \text{ for any } X\in\mathcal{X}(\alpha,\delta, A).
\end{align}
Firstly, observe that if $z' \in B_\delta(z)$, then $z^\delta \subset (z')^{3\delta}$. Thus,
\[ \left(\sum_{z\in X} \left(\frac{1}{\delta} \int_{z^\delta} |f|\right)^q  \right)^{1/q} \les \left(\sum_{z\in X} \frac{1}{|B_\delta(z)|} \int_{B_\delta(z)} \left( \frac{1}{\delta}\int_{z'^{3\delta}} |f| \right)^{q}  dz' \right)^{1/q} . \]
Let $\varphi_\delta(x) = \delta^{-n} \varphi(x/\delta)$ for an integrable function $\varphi\geq 1_{B_{10}}$. Then we have $\delta^{-1} \int_{z'^{3\delta}} |f|  \les \avg(|f|*\varphi_\delta)(z')$ (c.f. \eqref{eqn:conv}), which implies that
\begin{equation}\label{eq:du}
 \left(\sum_{z\in X} \left(\frac{1}{\delta} \int_{z^\delta} |f|\right)^q  \right)^{1/q}\les \delta^{-(n+1)/q} \| \avg(|f|*\varphi_\delta) \|_{L^q(X^\delta)}. 
\end{equation} 

We claim that, for a sufficiently small absolute constant $c_0>0$, the measure $\nu$ defined by 
\[ d\nu = c_0 (A\delta^{n+1})^{-1} 1_{X^\delta}(x,t) \, dx\, dt \]
belongs to $\mathcal{C}(\alpha)$. Indeed, for $\delta<r\leq 1$, we have
\begin{align*}
\int_{B_r} (A\delta^{n+1})^{-1} 1_{X^\delta} \les (A\delta^{n+1})^{-1} \delta^{n+1} \#(X \cap B_{r+\delta}) \les (r+\delta)^{\alpha} \les r^\alpha.
\end{align*}
For $0<r<\delta$, 
\[ \int_{B_r} (A\delta^{n+1})^{-1} 1_{X^\delta} \les (A\delta^{n+1})^{-1} r^{n+1} \les r^\alpha (r/ \delta)^{n+1-\alpha} \leq r^\alpha.\]
For the last inequality, we used the fact that for any non-empty $X$, there exists a ball $B_\delta$ such that $1 \leq \#(X\cap B_\delta)  \leq A \delta^{\alpha}$. 

We apply \eqref{eq:fullaverage} with the measure $\nu$, which yields 
\[ \| \avg(|f|*\varphi_\delta) \|_{L^{\tilde q}(X^\delta)} \les (A\delta^{n+1})^{1/{\tilde q}}  \| |f|*\varphi_\delta \|_{L^{\tilde p}} \les (A\delta^{n+1})^{1/{\tilde q}}  \| f \|_{L^{\tilde p}}. \]
This inequality implies, by \eqref{eq:du},
\begin{equation}\label{eq:du1}
 \left(\sum_{z\in X} \left(\frac{1}{\delta} \int_{z^\delta} |f|\right)^{\tilde{q}}  \right)^{1/\tilde{q}} \les A^{1/\tilde{q}}  \| f \|_{L^{\tilde p}}
\end{equation} 
giving, in view of the reduction from \eqref{eq:dualstep}, the desired inequality \eqref{eq:dual} for the pair of exponents $(\tilde{p},\tilde{q})$ rather than $(p,q)$. However, this defect can be remedied by an interpolation argument similar to the one used for the implication $\eqref{eq:dual1} \implies \eqref{eq:dual}$, since we may take $(\tilde{p},\tilde{q})$ arbitrarily close to $(p,q)$ and we allow a loss of $\delta^{-\epsilon}$.
\end{proof}

\section{Lower bounds}\label{sec:LowerBound}

\subsection{Lower bounds for $\NM$}

In this section, we prove \Cref{theorem:NM-upper-bounds}\ref{theorem:NM-upper-bounds-item:sharpness}.

\begin{proposition} \label{prop:NDeltaLowerBd}
Let $1 \leq p \leq \infty$ and $0<\delta<1$. Then
\begin{align*}
\| \NM \|_{L^p \to L^p} 
&\ges 
\max(
\delta^{1 - \frac{2}{p}}
,
1
)
&(n=2)
\\
\| \NM \|_{L^p \to L^p} &\ges 
\max(
\delta^{\frac{3}{2}-\frac{5}{2p}} 
,
\delta^{\frac{1}{2}-\frac{1}{p}}
,
1)
&(n=3)
\\
\| \NM \|_{L^p \to L^p} &\ges 
\max(
\delta^{2-\frac{3}{p}},
\delta^{\frac{1}{2}-\frac{1}{p}}, 
1
)
&(n\geq 4)
\end{align*}
\end{proposition}
\begin{proof}
To obtain each lower bound in \Cref{prop:NDeltaLowerBd}, we will specify a pair of sets $E, Z \subset \R^n$ satisfying
\begin{align}
\label{eq:EZ-good}
\text{for all } z \in Z^\delta,
\qquad
|E^{\delta} \cap S^\delta(z)| \gtrsim |E^\delta|.
\end{align}
We view $E$ as the ``test set'' and $Z$ as the ``set of centers.''

Suppose we have sets $E$ and $Z$ satisfying \eqref{eq:EZ-good}. Then $\NM 1_{E^\delta}(x) \gtrsim \delta^{-1}|E^\delta|$ for all $x \in Z^\delta+\S^{n-1}$, so
\[
\|\NM \|_{L^p\to L^p}
\geq
\frac{\|\NM 1_{E^\delta} \|_p}{\|1_{E^\delta}\|_p}
\gtrsim
\delta^{-1} |E^\delta|^{1-1/p} |Z^\delta+\S^{n-1}|^{1/p}.
\]
Thus, 
\begin{align}
\label{eq:alpha-beta-gamma-def}
\text{if $|E^\delta| \gtrsim \delta^{\alpha}$ and $|Z^\delta+\S^{n-1}| \gtrsim \delta^{\beta}$, then $\|\NM \|_{L^p\to L^p} \gtrsim \delta^\gamma$},
\end{align}
where $\gamma = (\alpha-1) - \frac{\alpha-\beta}{p}$.

The sets are given in \Cref{figure:examples-NM}. 
\end{proof}
\begin{figure}[h]
\begin{align*}
\renewcommand{\arraystretch}{1.7}
\begin{array}{|l|c|c|c|c|c|c|}
\hline
\text{name} & n \text{ (dim)} & E \text{ (test set)} & Z \text{ (set of centers)} & \alpha & \beta & \gamma  
\\
\hline
\hline
\text{sphere}
& \geq 2
& \S^{n-1}
& \{0\}^{n}
& 1
& 1
& 0
\\
\hline
\text{$\delta$-ball}
& 2
& \{(0,0)\}
& \S^1
& 2
& 0
& 1 - \frac{2}{p}
\\
\hline
\text{tube} 
& 3 
& \{(0,0)\} \times [0, \sqrt{\delta}]
& \S^{1} \times [0, \sqrt{\delta}]
& \frac{5}{2}
& 0
& \frac{3}{2} - \frac{5}{2p}
\\
\hline
\text{cylindrical shell}
& \geq 3
& [0, \sqrt{\delta}] \times \S^{n-2}
&  [0, \sqrt{\delta}] \times \{0\}^{n-1}
& \frac{3}{2}
& \frac{1}{2}
& \frac{1}{2}-\frac{1}{p}
\\
\hline
\text{radius $1/\sqrt{2}$}
& \geq 4
& \{(0,0)\} \times \frac{1}{\sqrt{2}} \S^{n-3}
& \frac{1}{\sqrt{2}} \S^{1} \times \{0\}^{n-2}
& 3
& 0
& 2-\frac{3}{p}
\\
\hline
\end{array}
\end{align*}
\caption{Table of examples for $\NM$. The sets $E$ and $Z$ satisfy \eqref{eq:EZ-good}. The quantities $\alpha, \beta, \gamma$ are defined as in \eqref{eq:alpha-beta-gamma-def}. (See \Cref{figure:cyl-shell-rad-1/sqrt2} for some pictures.)}
\label{figure:examples-NM}
\end{figure}

\begin{figure}[h!]
\centering
	\begin{subfigure}[b]{0.4\textwidth} \centering
\includegraphics[width=0.8\textwidth]{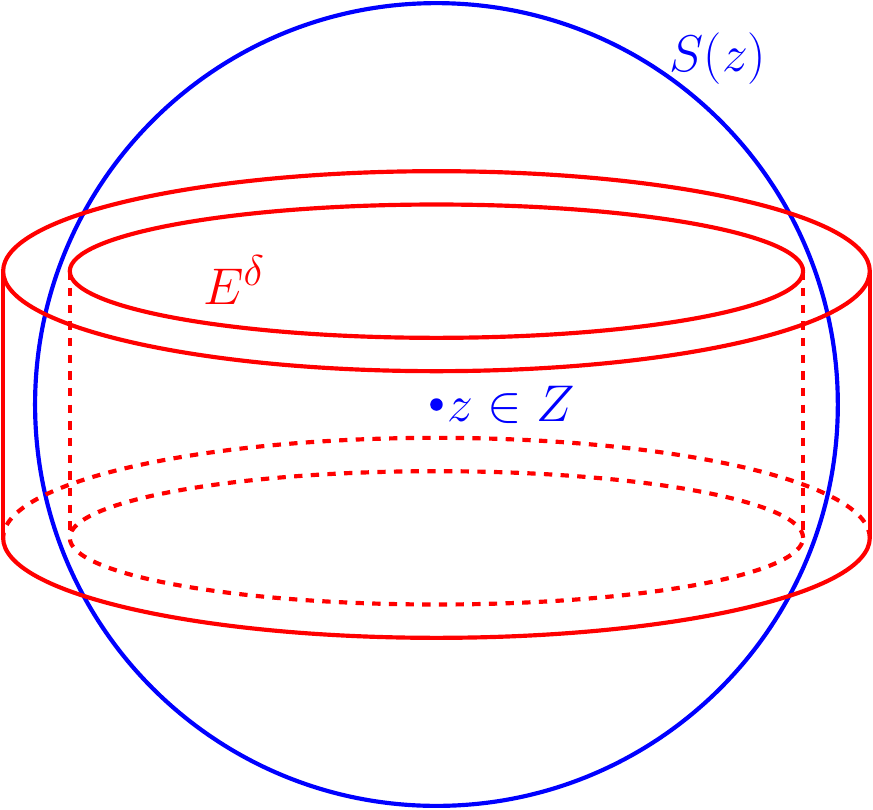}
\caption{Cylindrical shell example.}
	\end{subfigure}
\qquad
	\begin{subfigure}[b]{0.4\textwidth} \centering
\includegraphics[width=0.8\textwidth]{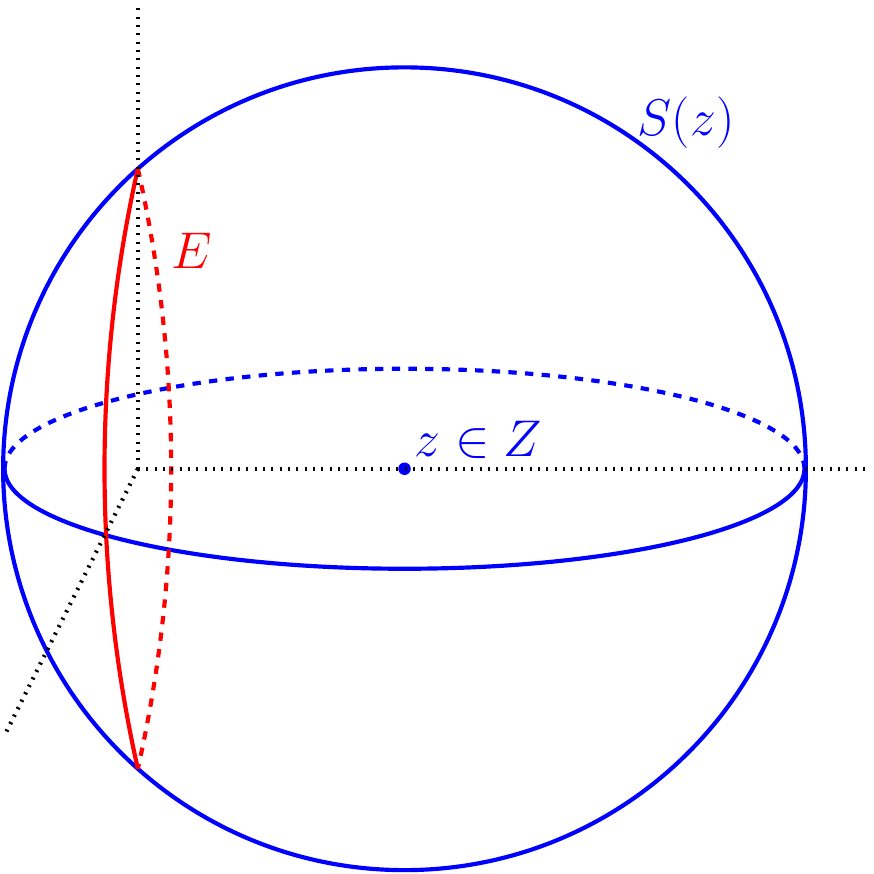}
\caption{Radius $1/\sqrt{2}$ example. (The horizontal axis represents $\R^2$, and the two remaining axes together represent $\R^{n-2}$).}
	\end{subfigure}
  \caption{The last two examples from \Cref{figure:examples-NM}.}
  \label{figure:cyl-shell-rad-1/sqrt2}
\end{figure}
\begin{remark}
\label{remark:lenz}
When $n=4$, the ``radius $1/\sqrt{2}$'' example is also known as the \emph{Lenz construction} in the context of the Erd\H{o}s unit distance problem. (See, e.g., \cite[Section 5.2]{BrassMoserPach}.) This type of example also appears in \cite[Proposition 2.1]{Kol_Wol}.
\end{remark}

\subsection{Lower bounds  for $\NS$}

In this section, we prove \Cref{theorem:NS-upper-bounds}\ref{theorem:NS-upper-bounds-item:sharpness}.

\begin{proposition} 
\label{prop:NSLowerBd}
Let $p \in [1,\infty]$ and $0<\delta<1$. Then \begin{align*}
\| \NS \|_{L^p \to L^p} &\ges \max(\delta^{\frac{1}{2} - \frac{3}{2p}},1) &(n=2)
\\
\| \NS \|_{L^p \to L^p} &\ges \max(\delta^{1-\frac{2}{p}},1) &(n\geq 3)
\end{align*}
\end{proposition}
\begin{proof}
The fact that $\| \NS \|_{L^p \to L^p} \gtrsim 1$ for all $p \in [1, \infty]$ follows from \Cref{prop:NDeltaLowerBd} and the fact that $\NS \geq \NM$, so we only need to consider the remaining two bounds  in \Cref{prop:NSLowerBd}.

To obtain each of these two lower bounds, we will specify a pair of sets $E, Z \subset \R^n$ and a function $r : Z \to [1,2]$ satisfying
\begin{align}
\label{eq:EZ-good2}
\text{for all } z \in Z, 
\qquad
|E^{\delta} \cap S^\delta(z,r(z))| \gtrsim |E^\delta|.
\end{align}
This implies that $\NS 1_{E^\delta}(x) \gtrsim \delta^{-1}|E^\delta|$ for all $x \in \bigcup_{z \in Z} S(z,r(z))$, so
\begin{align}
\label{eq:alpha-beta-gamma-def2}
\text{if $|E^\delta| \gtrsim \delta^{\alpha}$ and $\left|\bigcup_{z \in Z} S(z,r(z))\right| \gtrsim \delta^{\beta}$, then $\|\NS \|_{L^p\to L^p} \gtrsim \delta^\gamma$},
\end{align}
where $\gamma = (\alpha-1) - \frac{\alpha-\beta}{p}$. 

The sets are given in \Cref{figure:examples-NS}.
\end{proof}

\begin{figure}[h]
\begin{align*}
\renewcommand{\arraystretch}{1.7}
\begin{array}{|l|c|c|c|c|c|c|c|}
\hline
\text{name} & n \text{ (dim)} & E \text{ (test set)} & Z \text{ (set of centers)} & r(z) & \alpha & \beta & \gamma
\\
\hline
\hline
\text{tube} 
& 2 
& \{0\} \times   [0, \sqrt{\delta}]
& [1,2] \times [0, \sqrt{\delta}]
& z_1
& \frac{3}{2}
& 0
& \frac{1}{2} - \frac{3}{2p}
\\
\hline
\text{radius $1/\sqrt{2}$}
& \geq 3
& \{0\} \times \frac{1}{\sqrt{2}} \S^{n-2}
& [1,\frac{3}{2}] \times \{0\}^{n-1}
& \sqrt{z_1^2 + \frac{1}{2}}
& 2
& 0
& 1-\frac{2}{p}
\\
\hline
\end{array}
\end{align*}
\caption{Table of examples for $\NS$. The sets $E$ and $Z$ satisfy \eqref{eq:EZ-good2}. The quantities $\alpha, \beta, \gamma$ are defined as in \eqref{eq:alpha-beta-gamma-def2}.}
\label{figure:examples-NS}
\end{figure}

\subsection{Lower bounds for $\MT$}\label{sec:lowerboundMT}

In this section we prove \Cref{prop:lowerboundMT}. It is an immediate consequence of the following.
\begin{lemma}\label{lem:lowMT}
Let $n\geq 2$. For each $0\leq s \leq n-1$, let $\Gamma(n,p,s)$ be the set of all $\gamma \in \R$ such that there exists $E\subset \R^n$ and a compact set $T\subset \R^n$ with finite $s$-dimensional upper Minkowski content such that \[\frac{\| \MT 1_{E^\delta} \|_{L^p}}{\| 1_{E^\delta} \|_{L^p}} \ges \delta^{\gamma}\qquad\text{for all } \delta \in (0, 1/2).\] 
Then $\Gamma(n,p,s)$ contains the following numbers:
\begin{align*}
\begin{cases}
(n-1) - \frac{1}{p}(n-1+\min(s,1)), & 0\leq s \leq n-1  \\
  n-\frac{3}{2} - \frac{1}{p}\left(n-\frac{3}{2} +\frac{1}{2}\min(2,s) \right), &  1 < s \leq n-1 \\ 
  n-\ceil{s}+1 - \frac{1}{p}\left( n-2\ceil{s}+s+2 \right) , & 2 \leq s \leq n-1  \\
  n-\floor{s}+1 - \frac{1}{p}\left( n-\floor{s}+2 \right) , & 2 \leq s \leq n-1 \\
    \frac{1}{2} - \frac{1}{p}(2 + s - n), &  n-2 \leq s \leq n-1.
  \end{cases}
\end{align*}
\end{lemma}

\begin{proof}[Proof of \Cref{lem:lowMT}] 
We choose $T$ as follows:
\begin{align}
\label{eq:lowMT-def-T}
T =
\{0\}^{n - \ceil{s}}
\times 
C_s
\end{align}
where $C_s \subset \R^{\ceil{s}}$ is a self-similar $s$-dimensional Cantor-type set; more precisely, let $C_1 = [-\frac{1}{2},\frac{1}{2}]$, and for $0 < d < 1$, let $C_{d} \subset [-\frac{1}{2},\frac{1}{2}]$ be the standard symmetric Cantor set with Minkowski (and Hausdorff) dimension equal to $d$. (See, e.g., \cite[Section 4.10]{mattila95}.) Then for $s > 1$, let $C_s = (C_{s/\ceil{s}})^{\ceil{s}}$.

For four of the five cases in \Cref{lem:lowMT}, we will specify a pair of sets $E, Z \subset \R^n$ satisfying
\begin{align}
\label{eq:EZ-good-NT}
\text{for all } z \in Z^\delta,
\qquad
\cH^{n-1}(E^{\delta} \cap \S^{n-1}(z)) \gtrsim \delta^{-1}|E^\delta|.
\end{align}
This implies $\MT 1_{E^\delta}(x) \gtrsim \delta^{-1}|E^\delta|$ for all $x \in Z^\delta-T$, so
\begin{align}
\label{eq:alpha-beta-gamma-def-NT}
\text{if $|E^\delta| \gtrsim \delta^{\alpha}$ and $|Z^\delta-T| \gtrsim \delta^{\beta}$, then $\frac{\|\MT 1_{E^\delta} \|_p}{\|1_{E^\delta}\|_p}
 \gtrsim \delta^\gamma$},
\end{align}
where $\gamma = (\alpha-1) - \frac{\alpha-\beta}{p}$.

The sets are given in \Cref{figure:examples-NT}. These correspond to the first, second, third, and fifth cases of \Cref{lem:lowMT}, in that order. For the fourth case, we use monotonicity of $\Gamma(n,p,s)$ in $s$: since finite $s$-dimensional upper Minkowski content implies finite $s'$-dimensional Minkowski content for all $s' > s$, we have $\Gamma(n,p,s) \supset \Gamma(n,p,\floor{s}) \ni n-\floor{s}+1 - \frac{1}{p}( n-\floor{s}+2)$.
\end{proof}

\begin{figure}[h]
\footnotesize
\begin{align*}
\renewcommand{\arraystretch}{1.7}
\begin{array}{|l|c|c|c|c|c|}
\hline
\text{name} & \text{range of } s & E \text{ (test set)} & Z \text{ (set of centers)} & \alpha & \beta  
\\
\hline
\hline
\text{$\delta$-ball}
&[0, n-1]
& \{0\}^n
& \S^{n-1}
& n
& 1-\min(1,s)
\\
\hline
\text{tube} 
& [1, n-1]
& \{0\}^{n-1} \times  [0, \sqrt{\delta}] 
& \S^{n-2} \times [0, \sqrt{\delta}] 
& n-\frac12
& 1-\frac{1}{2} \min(2,s)
\\
\hline
\text{rad.~$1/\sqrt{2}$}
& [2, n-1]
&  \{0\}^{n-\lceil s\rceil+1} \times \frac{1}{\sqrt{2}} \S^{\lceil s \rceil - 2}
& \frac{1}{\sqrt{2}} \S^{n-\lceil s \rceil}   \times \{0\}^{\lceil s \rceil-1}
& n-\lceil s \rceil+2
& \lceil s\rceil - s
\\
\hline
\text{cyl.~shell}
& [n-2, n-1]
&  [0, \sqrt{\delta}] \times \S^{n-2}    
&  [0, \sqrt{\delta}] \times  \{0\}^{n-1}
& \frac32
& n - \frac{1}{2} - s
\\
\hline
\end{array}
\end{align*}
\caption{Table of examples for $\MT$. The sets $E$ and $Z$ satisfy \eqref{eq:EZ-good-NT}. The quantities $\alpha, \beta$ are defined as in \eqref{eq:alpha-beta-gamma-def-NT}, and $c > 0$ is a small absolute constant.}
\label{figure:examples-MT}
\label{figure:examples-NT}
\end{figure}

\subsection{Lower bounds for $\ST$}
\label{section:ST-lower}

In this section, we prove \Cref{cor:lower}. It is an immediate consequence of the following.

\begin{lemma}\label{lem:lowST}
Let $n\geq 2$. For each $0\leq s \leq n-1$, let $\Gamma(n,p,s)$ be the set of all $\gamma \in \R$ such that there exists $E\subset \R^n$ and a compact set $T\subset \R^n$ with finite $s$-dimensional upper Minkowski content such that \[\frac{\| \ST 1_{E^\delta} \|_{L^p}}{\| 1_{E^\delta} \|_{L^p}} \ges \delta^{\gamma}\qquad\text{for all } \delta \in (0, 1/2).\] 
Then $\Gamma(n,p,s)$ contains the following:
\begin{align*}
\begin{cases}
(n-\frac{\ceil{s}}{2}-1) - \frac{1}{p} (n-\ceil{s}+\frac{s}{2}) & s \leq  2
\\
(n-\frac{\floor{s}}{2}-1) - \frac{1}{p} (n-\frac{\floor{s}}{2})  & s \leq  2
\\
(n-\ceil{s}) - \frac{1}{p} (n+s-2\ceil{s}+1)
& s \geq 2
\\
(n-\floor{s}) - \frac{1}{p} (n-\floor{s}+1)
& s \geq 2.
\end{cases}
\end{align*}
\end{lemma}

\begin{proof}
We take $T$ as in \eqref{eq:lowMT-def-T}. For each of the five cases in \Cref{lem:lowST}, we will specify a pair of sets $E, Z \subset \R^n$ and a function $r : Z \to [1,2]$ satisfying
\begin{align}
\label{eq:EZ-good-ST}
\text{for all } z \in Z,
\qquad
\cH^{n-1}(E^{\delta} \cap \S^{n-1}(z,r(z))) \gtrsim \delta^{-1}|E^\delta|.
\end{align}
This implies $\ST 1_{E^\delta}(x) \gtrsim \delta^{-1}|E^\delta|$ for all $x \in \bigcup_{z \in Z} (z-r(z)T)$, so
\begin{align}
\label{eq:alpha-beta-gamma-def-ST}
\text{if $|E^\delta| \gtrsim \delta^{\alpha}$ and $\left|\bigcup_{z \in Z} (z-r(z)T)\right| \gtrsim \delta^{\beta}$, then $\frac{\|\ST 1_{E^\delta} \|_p}{\|1_{E^\delta}\|_p}
 \gtrsim \delta^\gamma$},
\end{align}
where $\gamma = (\alpha-1) - \frac{\alpha-\beta}{p}$.

The sets are given in \Cref{figure:examples-ST}. The calculation of $\beta$ here is less straightforward than in the other proofs in \Cref{sec:LowerBound} so far, so we provide some details for the ``tube'' example. We have
\[
z-r(z)T = (z_I, z_{II}) - |z_I| \{0\}^{n-\ceil{s}}
 \times C_s
=
\{z_I\} \times (z_{II} - |z_I| C_s)
\]
so
\begin{align*}
\bigcup_{z \in Z}
(z-r(z)T)
=
\bigcup_{1 \leq |z_I| \leq 2}
\{z_I\} \times
\left(
\bigcup_{|z_{II}| \leq \delta^{1/2}}
 (z_{II} - |z_I| C_s)
\right)
=
\bigcup_{1 \leq |z_I| \leq 2}
\{z_I\} \times
\left(
B_{\delta^{1/2}}^{\ceil{s}} - |z_I| C_s
\right).
\end{align*}
For each $z_I$, the set $B_{\delta^{1/2}}^{\ceil{s}} - |z_I| C_s$ is the $\delta^{1/2}$-neighborhood of a dilate of $C_s$, and thus has ($\ceil{s}$-dimensional) Lebesgue measure $\approx \delta^{(\ceil{s} - s)/2}$. Thus by Fubini, $|\bigcup_{z \in Z} (z-r(z)T)| \approx \delta^{(\ceil{s}-s)/2}$. The calculation of $\beta$ for the radius $1/\sqrt{2}$ example is similar.

For the second and fourth terms, we use monotonicity of $\Gamma$, i.e., $\Gamma(n,p,s) \supset \Gamma(n,p,\floor{s})$.
\end{proof}

\begin{figure}[h]
\footnotesize
\begin{align*}
\renewcommand{\arraystretch}{1.7}
\begin{array}{|l|c|c|c|c|c|c|}
\hline
\text{name} & \text{range of } s & E \text{ (test set)} & Z \text{ (set of centers)} & r(z) & \alpha & \beta  
\\
\hline
\hline
\text{``tube''}
& s\leq 2
&  \{0\}^{n-\ceil{s}} \times B_{\delta^{1/2}}^{\ceil{s}}
&  (B_2^{n-\ceil{s}} \setminus B_1^{n-\ceil{s}})  \times B_{\delta^{1/2}}^{\ceil{s}}
& |z_I|
& n - \frac{\ceil{s}}{2}
& \frac{\ceil{s}-s}{2}
\\
\hline
\text{rad.~$1/\sqrt{2}$}
& s\geq 2
& 
\{0\}^{n-\ceil{s}}
\times \frac{1}{\sqrt{2}} \S^{\ceil{s}-1}
&  
(B_{3/2}^{n-\ceil{s}} \setminus B_1^{n-\ceil{s}})  \times B_\delta^{\ceil{s}}
& \sqrt{|z_I|^2 + \frac{1}{2}}
& n - \ceil{s} + 1
& \ceil{s}-s
\\
\hline
\end{array}
\end{align*}
\caption{Table of examples for $\ST$. The sets $E$ and $Z$ satisfy \eqref{eq:EZ-good-ST}. The quantities $\alpha, \beta$ are defined as in \eqref{eq:alpha-beta-gamma-def-ST}. In both rows, $Z$ is a set of the form $Z_I \times Z_{II}$, and $z_I$ refers to the components in $Z_I$.}
\label{figure:examples-ST}
\end{figure}

\appendix
\section{The Kakeya needle problem for $\mathbb{S}^{n-1}$}\label{section:KakeyaSet}

Here, we sketch a proof of \Cref{theorem:nikodym-translations-sphere}. We begin with the following.

\begin{theorem}[Kakeya needle problem for spheres]
	\label{theorem:kakeya-translations-sphere}
	Let $\epsilon > 0$ be arbitrary. Then between the origin and any prescribed point in $\R^n$, there exists a polygonal path $P = \bigcup_{i=1}^m L_i$ with each $L_i$ a line segment, and for each $i$ there exists an $(n-1)$-plane $V_i $ containing $ 0$, such that
	\begin{equation}
	\label{eq:theorem-translations}
	\Big|\bigcup_{i}\bigcup_{p\in L_i} (p+\{x\in \S^{n-1}:\, \dist(x, V_i) > \epsilon\})\Big|< \epsilon.
	\end{equation}
\end{theorem}

The proof of \Cref{theorem:kakeya-translations-sphere} is a relatively straightforward generalization of the proof of \cite[Theorem 1.2]{CC2019}. Here, we provide a few details. We begin with the following two basic estimates, which are analogues of  \cite[Lemma 2.2]{CC2019} and \cite[Lemma 3.3]{CC2019}, respectively.

\begin{lemma}\label{lemma:kakeya-basic-estimate1}
For any polygonal path $P\subset \R^n$ and for an arbitrary $E\subset\R^n$, we have
$|P + E| \lesssim\cH^{n-1}(E)\cH^1(P)$. 
\end{lemma}

\begin{proof}
If $B$ is a ball of radius $r$, where $r$ is smaller than the line segments in the polygonal path, then for each line segment $L\subset P$, we have $|L+B|\lesssim r^{n-1}\cH^1(L)$. Adding these up for all line segments $L$ and approximating $E$ by a union of small balls, we obtain \Cref{lemma:kakeya-basic-estimate1}.
\end{proof}

\begin{lemma}
	\label{lemma:kakeya-basic-estimate2}
	Let $\epsilon > 0$. Let $L \subset \R^n$ be a line segment in the direction $\theta \in \P^{n-1}$. Suppose $R \subset \{x \in \S^{n-1} : |x \cdot \theta| \leq \epsilon\}$. Then $|L + R| \leq \epsilon \cH^1(L) \cH^{n-1}(R)$.
\end{lemma}

\begin{proof}
	Without loss of generality, assume $\theta = e_1 = (1,0, \ldots, 0)$. Let $P : \R^n \to \R$ be the orthogonal projection onto $e_1^\perp$. Let $R = R^+ \cup R^-$, where $R^+ = \{(x_1, \ldots, x_n) \in R : x_1 \geq 0\}$. Let $f : \R^{n-1} \to \R$ be given by $f(y) = \sqrt{1-|y|^2}$. Using the hypothesis on $R$, an elementary computation gives
\begin{align}
	\cH^{n-1}(R^+)
	=
	\int_{P(R^+)} \sqrt{1 + |\nabla f(y)|^2} \, dy
	\geq
	\frac{1}{\epsilon} \cH^{n-1}(P(R^+)).
	\end{align}
A similar inequality holds for $R^-$. By Fubini's theorem, we have
\[
     |L + R|
	\leq
	\cH^1(L) \left(\cH^{n-1}(P(R^+)) + \cH^{n-1}(P(R^-))\right)
	 \leq \epsilon \cH^1(L) \cH^{n-1}(R).
	\]
\end{proof}

\begin{proof}[Proof sketch for \Cref{theorem:kakeya-translations-sphere}]

We identify $\R^2$ as the subspace of $\R^n$ spanned by the first two standard basis vectors. By symmetry, we may assume that the prescribed point in the statement of \Cref{theorem:kakeya-translations-sphere} is in $\R^2$. Then we follow the iterated Venetian blind construction presented in \cite[Section 4]{CC2019}. This produces a path $P \subset \R^2 \subset \R^n$.

As noted in \cite[Remark 4.4]{CC2019}, the construction does not depend the set until \cite[Section 4.8]{CC2019}.  Starting from that point, we make the following changes:

\begin{enumerate}
    \item For an interval $I \subset \P^1$, we define \[E_I = \{x \in \S^{n-1} : \text{there exists $\theta \in I$ such that $x \cdot \theta = 0$}\}
    \]
    Here, $\P^1$ is the quotient of $\S^1$ obtained by identifying antipodal points together, and $\S^1$ is the unit circle in $\R^2 \subset \R^n$.

    \item We use \Cref{lemma:kakeya-basic-estimate1} and \Cref{lemma:kakeya-basic-estimate2} in place of  \cite[Lemma 2.2]{CC2019} and \cite[Lemma 3.3]{CC2019}, respectively.
\end{enumerate}

Note that if $I$ is an interval of the form $I = \P^1 \setminus B(\theta_0,\epsilon)$, then
$E_I = \{x \in \S^{n-1} : \dist(x,V) \geq \epsilon'\}$
where $V \subset \R^n$ is the linear hyperplane orthogonal to $\theta_0$, and $\epsilon'$ only depends on $\epsilon$.
\end{proof}

By considering the limit $\epsilon \to 0$ in the appropriate sense (see \cite[Section 6]{CC2019} for details), we can show that \Cref{theorem:kakeya-translations-sphere} implies the following. 

\begin{theorem}[Besicovitch set for spheres]
	\label{theorem:besicovitch-translations-sphere}
	For every path $P_0$ in $\R^n$  and for any neighborhood of $P_0$, there is a path $P$ in this neighborhood with the same endpoints as $P_0$ and there is a $(n-1)$-plane $V_p    $ containing $0$ such that
	\begin{equation}
	\label{eq:besicovitch-set-sphere}
	\Big|\bigcup_{p\in P} (p+ (\S^{n-1} \setminus V_p))\Big|=0.
	\end{equation}
Also, the mapping $p \mapsto V_p$ is Borel.
\end{theorem}

By taking countably many translates of the set in \eqref{eq:besicovitch-set-sphere} (again, see \cite[Section 6]{CC2019} for details), we can show \Cref{theorem:besicovitch-translations-sphere} implies \Cref{theorem:nikodym-translations-sphere}.

\section{Geometric estimates for annuli}\label{section:volumebound}
\subsection{Intersection of two annuli of radius comparable to 1}\label{sub:AppA}
Let $S(a,r)$ denote the sphere of radius $r$ centered at $a$ and let $S^\delta(a,r)$ denote its $\delta$-neighborhood. We define
\begin{align*}
d(S(a,r),S(b,s))  &= |a-b| + |r-s| \\
\Delta(S(a,r),S(b,s)) &= ||a-b|-|r-s|| \cdot |r+s-|a-b||.
\end{align*}

In \cite[Lemma 3.2]{Wolff_survey}, the following bound on $\R^2$ for circles is proved;
\begin{equation}\label{eqn:2d}
 |S^\delta(a,r) \cap S^\delta(b,s) | \les \frac{\delta^2}{\sqrt{(d+\delta)(\Delta+\delta)}}. 
\end{equation}
 Moreover, it was shown that the distance from the line through the centers and the intersection is $O(\sqrt{ \frac{\Delta+\delta}{d+\delta}})$. It is assumed there that the circles cannot be externally tangent; i.e. $|r+s-|a-b|| \sim 1$. An inspection of the proof shows that the estimate continues to hold without the assumption if we adopt the above definition of $\Delta(S(a,r),S(b,s))$.

We may generalize the above estimate in higher dimensions.
\begin{lemma} \label{lem:intersection}Let $n\geq 2$, $a,b\in \R^n$, and $r, s \in [1/2,2]$. Then 
\begin{align}
\label{eq:intersection-2-spheres-vary}
|S^\delta(a,r) \cap S^\delta(b,s) |\les
\frac{\delta^2}{d+\delta} \left(\frac{\Delta+\delta}{d+\delta}\right)^{(n-3)/2}.
\end{align}
As special cases, if $n \geq 3$, or if $n=2$ and $r=s$, then
\begin{align}
\label{eq:intersection-2-spheres-vary-special}
|S^\delta(a,r) \cap S^\delta(b,s) |\les
\frac{\delta^2}{d+\delta}.
\end{align}
\end{lemma}
\begin{proof}
The proof of \eqref{eq:intersection-2-spheres-vary} is by induction on $n$. When $n=2$, this is just \eqref{eqn:2d}. Without loss of generality, we may assume that the $x_1$-axis is the line through the centers of the spheres. It suffices to prove the inductive step
\begin{equation} \label{eqn:ind}
|S^\delta(a,r) \cap S^\delta(b,s) |  \les \sqrt{ \frac{\Delta+\delta}{d+\delta}}  |S^\delta(a,r) \cap S^\delta(b,s) \cap \{ x_n=0 \}|,
\end{equation}
where $|S^\delta(a,r) \cap S^\delta(b,s) \cap \{ x_n=0 \}|$ is the $(n-1)$-dimensional measure of the slice of the intersection.

We write $x' = (x_1,\cdots,x_{n-2})$ and use the change of variable 
$(x_1,\cdots,x_n) \mapsto (x', u\cos \theta, u\sin \theta)$ with $u\in \R$ and $\theta \in [0,\pi)$. Then
\[ |S^\delta(a,r) \cap S^\delta(b,s) |  = \int_0^{\pi} \iint_{\R^{n-1}}  1_{ S^\delta(a,y) \cap S^\delta(b,s)} (x', u\cos \theta, u\sin\theta) \, dx' \, |u| \, du\,  d\theta. \]
Note that $1_{ S^\delta(a,r) \cap S^\delta(b,s)} (x', u\cos \theta, u\sin\theta) = 1_{ S^\delta(a,r) \cap S^\delta(b,s)} (x', u, 0)$ and it is zero unless $|u| =O(\sqrt{ \frac{\Delta+\delta}{d+\delta}})$. This is because the distance from the line through the centers and the intersection is $O(\sqrt{ \frac{\Delta+\delta}{d+\delta}})$ as was observed in \cite{Wolff_survey}. This consideration proves \eqref{eqn:ind}.
\end{proof}

\subsection{Intersection of three annuli of radius 1}

\begin{lemma}[Intersection of three annuli]\label{lem:three_spheres}
There exist absolute constants $c_1, c_2>0$ so that the following is true. Suppose $n \geq 3$. Let $a_1, a_2, a_3 \in \R^n$ be distinct. Let $M = \max(|a_1-a_2|,|a_2-a_3|,|a_3-a_1|)$, let $m = \min(|a_1-a_2|,|a_2-a_3|,|a_3-a_1|)$, and let $R \in (0, \infty]$ denote the radius of the (unique) circle that passes through $a_1, a_2, a_3$.
Suppose $0 < \delta < \frac{1}{2}$ and 
\begin{align}
\label{eq:mM-bounds}
c_1\delta^{1/2} \leq m \leq M \leq c_2.
\end{align}
Then
\begin{align}
\label{eq:circumradius-near-1-estimate}
|\sph{\delta}{a_1} \cap \sph{\delta}{a_2} \cap \sph{\delta}{a_3}|
\lesssim
\frac{\delta^{5/2}}{M^{3/2}m^{1/2}}.
\end{align}
Furthermore, in two special cases, we have better bounds:
\begin{enumerate}
\item 
If $R \geq 2$, then
\begin{align}
\label{eq:empty-intersection}
\sph{\delta}{a_1} \cap \sph{\delta}{a_2} \cap \sph{\delta}{a_3} = \emptyset.
\end{align}
\item 
If $R \leq \frac{1}{2}$ or $n \geq 4$, then
\begin{align}
\label{eq:circumradius-small-estimate}
|\sph{\delta}{a_1} \cap \sph{\delta}{a_2} \cap \sph{\delta}{a_3}|
\lesssim
\frac{\delta^3}{M^2 m}.
\end{align}
\end{enumerate}
The implied constants depend only on the dimension $n$.
\end{lemma}

\begin{remark}
\label{remark:n-is-3-and-circumradius-close-to-1}
Note that \eqref{eq:empty-intersection} and \eqref{eq:circumradius-small-estimate} give better estimates than \eqref{eq:circumradius-near-1-estimate}. Thus,  \eqref{eq:circumradius-near-1-estimate} is only useful when $n=3$ and $\frac{1}{2} \leq R \leq 2$.
\end{remark}

\begin{proof}
Without loss of generality, we may make the following assumptions: $a_1, a_2, a_3 \in \R^2 \subset \R^n$, $|a_1| = |a_2| = |a_3| = R$, $|a_1-a_3| = M, |a_2-a_3| = m$, $a_1 = (-\frac{M}{2},(R^2 - \frac{M^2}{4})^{1/2})$, and $a_3 = (\frac{M}{2},(R^2 - \frac{M^2}{4})^{1/2})$.

First, we claim the following:
\begin{align}
\label{eq:triple-intersection-inclusion}
\begin{split}
&\sph{\delta}{a_1} \cap \sph{\delta}{a_2} \cap \sph{\delta}{a_3}    
\\
&\subset
\{(x,y) \in \R^2 \times \R^{n-2} : x \in S_1 \cap S_2 \cap B(a_1, 1+\delta), |y|^2 \in I(x)\}
\end{split}
\end{align}
where
\begin{align}
S_i &:= \{ x \in \R^2 : \left|\left< x, a_i-a_3\right> \right| \leq 2\delta\} \text{ for } i=1,2
\\
\label{eq:def-Ix}
I(x) &:= [(1-\delta)^2 - |x-a_1|^2, (1+\delta)^2 - |x-a_1|^2] \text{ for } x \in \R^2.
\end{align} 
Suppose $(x,y) \in \sph{\delta}{a_1} \cap \sph{\delta}{a_2} \cap \sph{\delta}{a_3}$ with $x \in \R^2$ and $y \in \R^{n-2}$. Then $(1-\delta)^2 \leq |x-a_i|^2 + |y|^2 \leq (1+\delta)^2$ for $i=1,2,3$. By subtracting these inequalities from each other, we obtain $x \in S_i$ for $i=1,2$. Furthermore, the inequality for $i=1$ implies $x \in B(a_1, 1+\delta)$ and $|y|^2 \in I(x)$. This completes the proof of \eqref{eq:triple-intersection-inclusion}, from which we deduce
\begin{align}
\begin{aligned}
&|\sph{\delta}{a_1} \cap \sph{\delta}{a_2} \cap \sph{\delta}{a_3}|
\\
&\labelrel\leq{eq:triple-intersection-inclusion}
\int_{S_1 \cap S_2 \cap B(a_1, 1+\delta)} | \{ y \in \R^{n-2} : |y|^2 \in I(x) \}| \, dx
\\
&\labelrel\lesssim{eq:def-Ix}
\delta
\int_{S_1 \cap S_2 \cap B(a_1, 1+\delta)} 
((1+\delta)^2-|x-a_1|^2)^{(n-4)/2}
\, dx
\label{eq:triple-intersection-fubini}    
\end{aligned}
\end{align}

Note that $S_i$ is an infinite strip of width $\frac{4\delta}{|a_i - a_3|}$. Let $\theta$ denote the angle between the strips $S_1$ and $S_2$. Then by elementary geometry,
$
S_1 \cap S_2 
\subset 
[-\frac{2\delta}{M}, \frac{2\delta}{M}] \times [-\frac{4\delta}{m\sin\theta}, \frac{4\delta}{m\sin\theta}].
$
Note that $\theta$ is also the angle of $a_3$ in the triangle $a_1,a_2,a_3$, so by the law of sines, $ \sin\theta =  \frac{|a_1-a_2|}{2R} \geq \frac{M}{4R}$. Thus, 
\begin{align}
\label{eq:S1-cap-S2}    
S_1 \cap S_2 
\subset 
[-\frac{2\delta}{M}, \frac{2\delta}{M}] \times [-\frac{16\delta R}{m M}, \frac{16\delta R}{m M}]
\end{align}

Now we prove \eqref{eq:empty-intersection}. Suppose $R \geq 2$. By choosing $c_1$ large enough in \eqref{eq:mM-bounds}, we have $S_1 \cap S_2 \subset B(0, \frac{R}{5})$. On the other hand, $x \in B(a_1, 1+\delta)$ implies
$|x - a_1| \leq 1+\delta \leq \frac{3}{2}$, so $|x| \geq |a_1| - |x-a_1| \geq R-\frac{3}{2} \geq \frac{R}{4}$. This shows $S_1 \cap S_2 \cap B(a_1, 1+\delta) = \emptyset$, which implies \eqref{eq:empty-intersection}.

Next, we prove \eqref{eq:circumradius-small-estimate}. We may assume $R \leq 2$. Suppose $n \geq 4$ or $R \leq \frac{1}{2}$. By \eqref{eq:triple-intersection-fubini} and \eqref{eq:S1-cap-S2}, it suffices to show
\begin{align}
\label{eq:n-4-over-2}
((1+\delta)^2-|x-a_1|^2)^{(n-4)/2} \lesssim 1
\qquad\text{for all } x \in S_1 \cap S_2 \cap B(a_1,1+\delta).
\end{align}
If $n \geq 4$, then the exponent $\frac{n-4}{2}$ is nonnegative, so \eqref{eq:n-4-over-2} holds. Thus, it remains to consider the case $n = 3$ and $R \leq \frac{1}{2}$. If $x \in S_1 \cap S_2$, then by choosing $c_2$ small enough in \eqref{eq:mM-bounds}, we have $|x| \leq \frac{1}{4}$, so $|x-a_1| \leq |x| + |a_1| \leq \frac{3}{4}$. Thus $(1+\delta)^2-|x-a_1|^2 \geq 1 - (\frac{3}{4})^2 > 0$, so we obtain \eqref{eq:n-4-over-2}.

Finally, we prove \eqref{eq:circumradius-near-1-estimate}. As noted in \Cref{remark:n-is-3-and-circumradius-close-to-1}, we may suppose $n = 3$ and $R \leq 2$. By \eqref{eq:triple-intersection-fubini}, it suffices to show
\begin{align}
\label{eq:goal-r-close-to-1}
\int_{S_1 \cap S_2 \cap B(a_1, 1+\delta)} 
((1+\delta)^2-|x-a_1|^2)^{-1/2}
\, dx
\lesssim
\frac{\delta^{3/2}}{M^{3/2} m^{1/2}} 
\end{align}
Write $x = (x',x'')$ and $a_1 = (a_1', a_1'')$. Define $v(x') = ((1+\delta)^2 - |x' - a_1'|^2)^{1/2}$, so that $x \in B(a_1, 1+\delta) \iff |x'' - a_1''| \leq v(x')$.  Suppose $x \in S_1 \cap S_2$. Then $|x'-a_1'| \leq |x'| + |a_1'| \leq \frac{2\delta}{M} +  \frac{M}{2} \leq \frac{1}{2}$, if $c_1$ is large enough and $c_2$ is small enough in \eqref{eq:mM-bounds}. This implies $v(x') \geq \frac{3}{4}$, so
\[
(1+\delta)^2-|x-a_1|^2
=
v(x')^2 - |x''-a_1''|^2
\gtrsim
v(x') - |x''-a_1''|.
\]

 Thus,
\begin{align}
&\int_{S_1 \cap S_2 \cap B(a_1, 1+\delta)} 
((1+\delta)^2-|x-a_1|^2)^{-1/2}
\, dx
\\
&\lesssim
\int_{-2\delta/M}^{2\delta/M} 
\int_{\substack{|x''| \leq 32\delta/(mM)\\|x''-a_1''| \leq v(x')}}
(v(x') - |x''-a_1''|)^{-1/2}
\, dx'' \, dx' \label{eq:int3}
\end{align}
(Above, we used \eqref{eq:S1-cap-S2} and $R \leq 2$.)
For a fixed $x'$, we split the inner integral into two parts, depending on the sign of $x'' - a_1''$. Each part, after a change of variables, is an integral of the form $\int_I t^{-1/2} \, dt$, where $I$ is an interval of length at most $\frac{64\delta}{mM}$. Thus,
\begin{align*}
\eqref{eq:int3}
\lesssim
\int_{-2\delta/M}^{2\delta/M} 
\int_{0}^{64\delta/(mM)} 
t^{-1/2}
\, dt \, dx'
\lesssim
\frac{\delta^{3/2}}{M^{3/2} m^{1/2}}.
\end{align*}
which proves \eqref{eq:goal-r-close-to-1}.
\end{proof}

\providecommand{\bysame}{\leavevmode\hbox to3em{\hrulefill}\thinspace}
\providecommand{\MR}{\relax\ifhmode\unskip\space\fi MR }
\providecommand{\MRhref}[2]{%
	\href{http://www.ams.org/mathscinet-getitem?mr=#1}{#2}
}
\providecommand{\href}[2]{#2}


\begin{thebibliography}{GWZ20}
	
	\bibitem[BD15]{BD}
	Jean Bourgain and Ciprian Demeter, \emph{The proof of the {$l\sp 2$} decoupling
		conjecture}, Ann. of Math. (2) \textbf{182} (2015), no.~1, 351--389.
	\MR{3374964}
	
	\bibitem[Bes28]{Besicovitch1928}
	A.~S. Besicovitch, \emph{On {K}akeya's problem and a similar one}, Math. Z.
	\textbf{27} (1928), no.~1, 312--320. \MR{1544912}
	
	\bibitem[BHS21]{Beltran-Hickman-Sogge-exposition}
	David Beltran, Jonathan Hickman, and Christopher~D. Sogge, \emph{Sharp local
		smoothing estimates for fourier integral operators}, Geometric Aspects of
	Harmonic Analysis (Cham) (Paolo Ciatti and Alessio Martini, eds.), Springer
	International Publishing, 2021, pp.~29--105.
	
	\bibitem[BMP05]{BrassMoserPach}
	Peter Brass, William Moser, and J\'{a}nos Pach, \emph{Research problems in
		discrete geometry}, Springer, New York, 2005. \MR{2163782}
	
	\bibitem[Bou85]{Bou_Max}
	Jean Bourgain, \emph{Estimations de certaines fonctions maximales}, C. R. Acad.
	Sci. Paris S\'{e}r. I Math. \textbf{301} (1985), no.~10, 499--502.
	\MR{812567}
	
	\bibitem[CC19]{CC2019}
	Alan Chang and Marianna Cs\"{o}rnyei, \emph{The {K}akeya needle problem and the
		existence of {B}esicovitch and {N}ikodym sets for rectifiable sets}, Proc.
	Lond. Math. Soc. (3) \textbf{118} (2019), no.~5, 1084--1114. \MR{3946717}
	
	\bibitem[CHL17]{CHL}
	Chu-Hee Cho, Seheon Ham, and Sanghyuk Lee, \emph{Fractal {S}trichartz estimate
		for the wave equation}, Nonlinear Anal. \textbf{150} (2017), 61--75.
	\MR{3584933}
	
	\bibitem[Cor77]{cordoba1977}
	Antonio Cordoba, \emph{The {K}akeya maximal function and the spherical
		summation multipliers}, Amer. J. Math. \textbf{99} (1977), no.~1, 1--22.
	\MR{447949}
	
	\bibitem[Cun71]{cunningham}
	F.~Cunningham, Jr., \emph{The {K}akeya problem for simply connected and for
		star-shaped sets}, Amer. Math. Monthly \textbf{78} (1971), 114--129.
	\MR{0275287 (43 \#1044)}
	
	\bibitem[Eg04]{Erdogan}
	M.~Burak Erdo\~{g}an, \emph{A note on the {F}ourier transform of fractal
		measures}, Math. Res. Lett. \textbf{11} (2004), no.~2-3, 299--313.
	\MR{2067475}
	
	\bibitem[GWZ20]{GWZ}
	Larry Guth, Hong Wang, and Ruixiang Zhang, \emph{A sharp square function
		estimate for the cone in {$\Bbb {R}^3$}}, Ann. of Math. (2) \textbf{192}
	(2020), no.~2, 551--581. \MR{4151084}
	
	\bibitem[Har18]{Harris0}
	Terence L.~J. Harris, \emph{Refined strichartz inequalities for the wave
		equation}, arXiv:1805.07146 (2018).
	
	\bibitem[Har19]{Harris}
	\bysame, \emph{Improved decay of conical averages of the {F}ourier transform},
	Proc. Amer. Math. Soc. \textbf{147} (2019), no.~11, 4781--4796. \MR{4011512}
	
	\bibitem[HKL22]{HKL}
	Seheon Ham, Hyerim Ko, and Sanghyuk Lee, \emph{Circular average relative to
		fractal measures}, Communications on Pure and Applied Analysis \textbf{21}
	(2022), no.~10, 3283--3307.
	
	\bibitem[HL16]{HL2016}
	K.~H\'era and M.~Laczkovich, \emph{The {K}akeya problem for circular arcs},
	Acta Math. Hungar. \textbf{150} (2016), no.~2, 479--511. \MR{3568105}
	
	\bibitem[KW99]{Kol_Wol}
	Lawrence Kolasa and Thomas Wolff, \emph{On some variants of the {K}akeya
		problem}, Pacific J. Math. \textbf{190} (1999), no.~1, 111--154. \MR{1722768}
	
	\bibitem[Mar87]{Marstrand1987PackingCI}
	John~M. Marstrand, \emph{Packing circles in the plane}, Proceedings of The
	London Mathematical Society (1987), 37--58.
	
	\bibitem[Mat95]{mattila95}
	Pertti Mattila, \emph{Geometry of sets and measures in {E}uclidean spaces},
	Cambridge Studies in Advanced Mathematics, vol.~44, Cambridge University
	Press, Cambridge, 1995, Fractals and rectifiability. \MR{1333890}
	
	\bibitem[Mat15]{Mattila2015}
	P.~Mattila, \emph{Fourier analysis and {H}ausdorff dimension}, Cambridge
	Studies in Advanced Mathematics, Cambridge University Press, 2015.
	
	\bibitem[Mit99]{Mitsis}
	T.~Mitsis, \emph{On a problem related to sphere and circle packing}, J. London
	Math. Soc. (2) \textbf{60} (1999), no.~2, 501--516. \MR{1724841}
	
	\bibitem[MSS92]{MSS_Bo}
	Gerd Mockenhaupt, Andreas Seeger, and Christopher~D. Sogge, \emph{Wave front
		sets, local smoothing and {B}ourgain's circular maximal theorem}, Ann. of
	Math. (2) \textbf{136} (1992), no.~1, 207--218. \MR{1173929}
	
	\bibitem[Nik27]{nikodym1927}
	Otton Nikodym, \emph{Sur la mesure des ensembles plans dont tous les points
		sont rectilin{\'e}airement accessibles}, Fundamenta Mathematicae \textbf{10}
	(1927), no.~1, 116--168.
	
	\bibitem[Obe06]{Oberlin}
	Daniel~M. Oberlin, \emph{Packing spheres and fractal {S}trichartz estimates in
		{$\Bbb R^d$} for {$d\geq 3$}}, Proc. Amer. Math. Soc. \textbf{134} (2006),
	no.~11, 3201--3209. \MR{2231903}
	
	\bibitem[PYZ22]{Pramanik-Yang-Zahl}
	Malabika Pramanik, Tongou Yang, and Joshua Zahl, \emph{A {F}urstenberg-type
		problem for circles, and a {K}aufman-type restricted projection theorem in
		$\mathbb{R}^3$}, arXiv:2207.02259 (2022).
	
	\bibitem[Sch97]{Schlag_Bo}
	W.~Schlag, \emph{A generalization of {B}ourgain's circular maximal theorem}, J.
	Amer. Math. Soc. \textbf{10} (1997), no.~1, 103--122. \MR{1388870}
	
	\bibitem[Sch03]{Schlag_incidence}
	\bysame, \emph{On continuum incidence problems related to harmonic analysis},
	J. Funct. Anal. \textbf{201} (2003), no.~2, 480--521. \MR{1986697}
	
	\bibitem[Sog91]{Sogge_lo}
	Christopher~D. Sogge, \emph{Propagation of singularities and maximal functions
		in the plane}, Invent. Math. \textbf{104} (1991), no.~2, 349--376.
	\MR{1098614}
	
	\bibitem[Sog17]{Sogge}
	\bysame, \emph{Fourier integrals in classical analysis}, second ed., Cambridge
	Tracts in Mathematics, vol. 210, Cambridge University Press, Cambridge, 2017.
	\MR{3645429}
	
	\bibitem[Ste76]{Ste_Max}
	Elias~M. Stein, \emph{Maximal functions. {I}. {S}pherical means}, Proc. Nat.
	Acad. Sci. U.S.A. \textbf{73} (1976), no.~7, 2174--2175. \MR{420116}
	
	\bibitem[Ste93]{Stein}
	\bysame, \emph{Harmonic analysis: real-variable methods, orthogonality, and
		oscillatory integrals}, Princeton University Press, Princeton, NJ, 1993.
	\MR{1232192 (95c:42002)}
	
	\bibitem[Wol97]{Wolff_Circ}
	Thomas Wolff, \emph{A {K}akeya-type problem for circles}, Amer. J. Math.
	\textbf{119} (1997), no.~5, 985--1026. \MR{1473067}
	
	\bibitem[Wol99a]{Wolff_decay}
	\bysame, \emph{Decay of circular means of {F}ourier transforms of measures},
	Internat. Math. Res. Notices (1999), no.~10, 547--567. \MR{1692851}
	
	\bibitem[Wol99b]{Wolff_survey}
	\bysame, \emph{Recent work connected with the {K}akeya problem}, Prospects in
	mathematics ({P}rinceton, {NJ}, 1996), Amer. Math. Soc., Providence, RI,
	1999, pp.~129--162. \MR{1660476}
	
	\bibitem[Wol00]{WoL}
	\bysame, \emph{Local smoothing type estimates on {$L\sp p$} for large {$p$}},
	Geom. Funct. Anal. \textbf{10} (2000), no.~5, 1237--1288. \MR{1800068}
	
	\bibitem[Wol03]{wolff2003lectures}
	\bysame, \emph{Lectures on harmonic analysis}, University Lecture Series,
	vol.~29, American Mathematical Society, Providence, RI, 2003, With a foreword
	by Charles Fefferman and a preface by Izabella \L aba, Edited by \L aba and
	Carol Shubin. \MR{2003254}
	
	\bibitem[Zah12a]{Zahl_Circular}
	Joshua Zahl, \emph{{$L^3$} estimates for an algebraic variable coefficient
		{W}olff circular maximal function}, Rev. Mat. Iberoam. \textbf{28} (2012),
	no.~4, 1061--1090. \MR{2990134}
	
	\bibitem[Zah12b]{Zahl}
	\bysame, \emph{On the {W}olff circular maximal function}, Illinois J. Math.
	\textbf{56} (2012), no.~4, 1281--1295. \MR{3231483}
	
	\bibitem[Zyg02]{Zygmund}
	A.~Zygmund, \emph{Trigonometric series. {V}ol. {I}, {II}}, third ed., Cambridge
	Mathematical Library, Cambridge University Press, Cambridge, 2002, With a
	foreword by Robert A. Fefferman. \MR{1963498}
	
\end{thebibliography}
\end{document}